\documentclass[10pt,reqno]{amsproc}

\usepackage{fullpage}
\usepackage[dvipsnames]{xcolor}

\usepackage{
  amsmath, amsthm, amssymb, mathtools, dsfont, units,          
  graphicx, wrapfig, subfig, float,                            
  listings, color, inconsolata, pythonhighlight,               
  fancyhdr,
  hyperref,
  enumerate,
  enumitem,
  framed
}   

\usepackage{hyperref, lipsum} 
\usepackage{bm}
\usepackage{subfig}

\usepackage{newpxtext, newpxmath, inconsolata}
\usepackage{amsfonts}
\usepackage{pgfplots}
\pgfplotsset{compat=1.12}
\usepackage{tkz-fct}
\usepackage{svg}
\usepackage{tikz}
\usepackage{tikz-cd}
\usepackage{lipsum}
\usepackage{rank-2-roots}
\usepackage[utf8]{inputenc}
\usepackage{multicol}
\usepackage{multirow}

\hypersetup{colorlinks=true, linkcolor=red, citecolor=blue, filecolor=blue, urlcolor=blue}

\DeclareUnicodeCharacter{3BC}{$\pi$}
\DeclareUnicodeCharacter{3C0}{$\pi$}

\usepackage[bottom]{footmisc}

\allowdisplaybreaks

\usepackage[font={it,footnotesize}]{caption}

\usepackage{titlesec}
\titleformat{\section}{\large\bfseries}{\thesection\;\;\;}{0em}{}
\titleformat{\subsection}{\normalsize\bfseries\selectfont}{\thesubsection\;\;\;}{0em}{}

\setlist[itemize]{wide=0pt, leftmargin=16pt, labelwidth=10pt, align=left}

\graphicspath{{Images/}{../Images/}}

\usepackage[mathscr]{euscript}

\newcommand{\what}[1]{\widehat{#1}}

\newcommand{\SL}{\mathrm{SL}}

\newcommand{\QM}{\mathcal{QM}}

\newcommand{\dd}{\mathrm{d}}

\newcommand{\rE}{\mathrm{E}}

\newcommand{\bC}{\mathbb{C}}

\newcommand{\bR}{\mathbb{R}}

\newcommand{\bZ}{\mathbb{Z}}

\newcommand{\cE}{\mathcal{E}}

\newcommand{\cH}{\mathcal{H}}

\newcommand{\cL}{\mathcal{L}}
\newcommand{\cM}{\mathcal{M}}

\theoremstyle{definition}
\newtheorem{theorem}{Theorem}[section]
\newtheorem{proposition}[theorem]{Proposition}
\newtheorem{lemma}[theorem]{Lemma}
\newtheorem{corollary}[theorem]{Corollary}
\newtheorem{definition}[theorem]{Definition}
\newtheorem{example}[theorem]{Example}
\newtheorem{question}[theorem]{Question}
\newtheorem{conjecture}[theorem]{Conjecture}

\theoremstyle{remark}
\newtheorem{remark}[theorem]{Remark}

\newtheorem*{theorem*}{Theorem}
\newtheorem*{proposition*}{Proposition}
\newtheorem*{lemma*}{Lemma}
\newtheorem*{corollary*}{Corollary}
\newtheorem*{definition*}{Definition}
\newtheorem*{example*}{Example}
\newtheorem*{remark*}{Remark}

\newtheorem*{conjecture*}{Conjecture}

\usepackage{fancyhdr}
\newenvironment{red}{\relax\color{red}}{\hspace*{.5ex}\relax}
\newenvironment{blue}{\relax\color{blue}}{\hspace*{.5ex}\relax}
\newcommand{\ber}{\begin{red}}
\newcommand{\er}{\end{red}}
\newcommand{\beb}{\begin{blue}}
\newcommand{\eb}{\end{blue}}

\title{Inequalities involving polynomials and quasimodular forms}
\author{Seewoo Lee}
\date{}

\begin{document}

\maketitle

\begin{abstract}
In this paper, we study inequalities involving polynomials and quasimodular forms.
More precisely, we focus on the monotonicity of the functions of the form $t \mapsto t^m F(it)$ where $F$ is a quasimodular form and $m > 0$.
As an application, we construct infinitely many positive quasimodular forms of level $> 1$.
We also give alternative proofs of modular form inequalities used in the proof of optimality of Leech lattice packing and universal optimality of the lattice by Cohn, Kumar, Miller, Radchenko, and Viazovska.
\end{abstract}

\section{Introduction}

In 2016, Viazovska gave a surprisingly short and elegant proof of the optimality of $\rE_8$-sphere packing in dimension $8$ \cite{viazovska2017sphere}, based on the Cohn--Elkies' linear programming bound \cite{cohn2003new}.
Due to its simplicity, it took only one extra week to solve the problem for dimension 24, with Cohn, Kumar, Miller, and Radchenko \cite{cohn2017sphere}.
The main idea is to construct \emph{magic functions} that satisfy the desired properties for the Cohn--Elkies bound, and they could achieve this using \emph{quasimodular forms}.
In particular, the required positivity of the magic functions reduces to certain quasimodular form inequalities.
Later, the same group of people proved the stronger claim; the $\rE_8$ and Leech lattice are \emph{universally optimal}, i.e. they are minimizers of a large class of potential functions among all possible configurations in 8 and 24 dimensions \cite{cohn2022universal}.
The proof is based on the Fourier interpolation formula, where the authors need to prove inequalities involving quasimodular forms and polynomials that are considerably harder than the inequalities in the sphere packing paper.

The original proofs of the inequalities by Viazovska and Cohn et al. are based on rigorous numerical analysis, and it is natural to ask if there are more conceptual proofs of them.
Romik \cite{romik2023viazovska} gave such a proof for the inequalities for the 8-dimensional sphere packing, and the author gave alternative proofs for both 8 and 24-dimensional packings \cite{lee2024algebraic} which are algebraic in nature.
We may also ask if the inequalities for the universal optimality can be proved in a similar way, which is one of the main motivations of the paper.
The simplest but nontrivial inequality in \cite[p. 1067]{cohn2022universal} is (3b): if
\begin{align}
    \what{\cE}_{1, 0} &= 12 \pi i (-E_2 E_4 E_6 + E_6^2 + 720 \Delta) \\
    \what{\cE}_{1, 1} &= 2 \pi^2 (E_2^2 E_4 E_6 - 2 E_2 E_6^2 - 1728 E_2 \Delta + E_4^2 E_6) \\
    \what{\cE}_{1}(\tau) &= \tau \what{\cE}_{1, 1}(\tau) + \what{\cE}_{1, 0}(\tau),
\end{align}
then the inequality reads as
\begin{equation}
    \label{eqn:inequality3b}
    i \what{\cE_{1}}(it) > 0
\end{equation}
for all $t > 0$.
Here $E_2, E_4, E_6$ are Eisenstein series and $\Delta$ is the discriminant form (see \ref{subsec:prelim_qmf}).

In this paper, instead of only focusing on giving a new proof of \eqref{eqn:inequality3b}, we develop a more general framework to prove similar inequalities.
More precisely, we study the monotonicity of the functions of the form
\begin{equation}
    t \mapsto t^m F(it) \label{eqn:tmf}
\end{equation}
for a quasimodular form $F$ and $m > 0$, or equivalently, the positivity of the function
\begin{equation}
    t \mapsto - F(it) + \frac{2 \pi t}{m} F'(it). \label{eqn:tmf_deriv}
\end{equation}

We give several applications of the monotonicity of such functions.
First of all, it can be used to construct positive quasimodular forms of level $> 1$ (which are not completely positive in general, see Corollary \ref{cor:polymod_posdif}), and we use them to give another proof of one of the quasimodular form inequalities for the optimality of Leech lattice packing \cite{cohn2017sphere} (see Section \ref{subsec:application_packingineq}).
Also, we show that the inequality (3b) used in the proof of the universal optimality of the Leech lattice \cite{cohn2022universal} is equivalent to the monotonicity of \eqref{eqn:tmf} for $m = 11$ and $F = X_{12, 1}$, the extremal quasimodular form of weight $12$ and depth $1$ (see Section \ref{subsec:prelim_extqmf} for the definition), where the monotonicity of the corresponding function \eqref{eqn:tmf} can be proven by a general theorem (Proposition \ref{prop:polymod_tangent}).
Inspired from the above examples, we also propose several conjectures on the monotonicity when $F$ is another extremal quasimodular form (Conjecture \ref{conj:Xw1mono}) and the density of the positive coefficients of positive quasimodular forms (Conjecture \ref{conj:density}).
Accompanying Sage and Lean codes that checks certain identites and inequalities are available at the GitHub repository \url{https://github.com/seewoo5/posqmf}.

\subsection*{Acknowledgements}

We would like to thank Ken Ono, Sug Woo Shin and Maryna Viazovska for helpful comments.
Note that the proof of Lemma \ref{lem:X101weak} (in particular, the inequality \eqref{eqn:fpos}) was found by automated tools (OpenAI's ChatGPT5.2-Pro and AxiomMath's AxiomProver), with the latter providing formal verification in Lean along with the two inequalities \eqref{eqn:appendix_ineq1} and \eqref{eqn:appendix_ineq2}.

\section{Preliminaries}\label{sec:preliminaries}

\subsection{Modular forms and quasimodular forms}
\label{subsec:prelim_qmf}

\emph{Modular form} of weight $w$ and level $\Gamma \subseteq \SL_2(\bZ)$ is a holomorphic function $f:\cH \to \bC$ on the upper half-plane $\cH = \{ z \in \bC : \Im(z) > 0\}$ satisfying the transformation property
\[
f\left(\frac{az+b}{cz+d}\right) = (cz+d)^w f(z) \quad \text{for all } \begin{pmatrix} a & b \\ c & d \end{pmatrix} \in \Gamma
\]
and analytic conditions at cusps of $\Gamma$ (see e.g. \cite{diamond2005first}).
When $\left(\begin{smallmatrix}
    1 & 1 \\ 0 & 1
\end{smallmatrix}\right) \in \Gamma$, $f(z + 1) = f(z)$ and $f$ admits a Fourier expansion $f(z) = \sum_{n \ge 0} a_n q^n$ for $q = e^{2 \pi i z}$.
The standard examples include the Eisenstein series and the discriminant form
\begin{align}
    E_4(z) &= 1 + 240 \sum_{n \ge 1} \sigma_3(n) q^n \\
    E_6(z) &= 1 - 504 \sum_{n \ge 1} \sigma_5(n) q^n \\
    \Delta(z) &= \frac{E_4(z)^3 - E_6(z)^2}{1728} = q\prod_{n \ge 1}(1 - q^n)^{24} = \sum_{n \ge 1} \tau(n) q^n
\end{align}
where $\tau(n)$ is the Ramanujan tau function.
These modular forms have level $\SL_2(\bZ)$ and weight $4, 6, 12$, respectively, and the ring of modular forms of level $\SL_2(\bZ)$ is generated by $E_4$ and $E_6$.

\emph{Quasimodular forms} are holomorphic functions that satisfy ``almost-modular'' properties, i.e. not exactly invariant under the slash action, but the error term can be expressed in terms of a combination of rational terms and modular forms.
More precisely, a holomorphic function $F:\cH \to \bC$ is a \emph{quasimodular form} of weight $w$, depth $s$, and level $\Gamma$ if there exists holomorphic functions $F_0, F_1, \dots, F_s$ such that
\[
(F|_w \gamma)(z) := (cz + d)^{-w} F\left(\frac{az + b}{cz + d}\right) = \sum_{r=0}^{s} F_r(z)\left(\frac{c}{cz + d}\right)^r, \quad \text{for all }\gamma = \begin{pmatrix}
    a & b \\ c & d
\end{pmatrix} \in \Gamma.
\]
(See Section 5.3 of \cite{bruinier2008elliptic} for more details and applications of quasimodular forms.)
The most important example is the weight 2 Eisenstein series, which can be defined in terms of Fourier expansion:
\begin{equation}
    E_2(z) = 1 - 24 \sum_{n \ge 1} \sigma_1(n) q^n \label{eqn:e2fourier}
\end{equation}
which satisfies the following equation
\begin{equation}
    E_2\left(\frac{az + b}{cz + d}\right) = (cz + d)^2 E_2(z) + \frac{6c}{\pi i} (cz + d) \quad \text{for all } \begin{pmatrix} a & b \\ c & d \end{pmatrix} \in \SL_2(\bZ). \label{eqn:e2_transform}
\end{equation}
for any $z \in \cH$ and $\left(\begin{smallmatrix}
    a & b \\ c & d
\end{smallmatrix}\right) \in \SL_2(\bZ)$.
In particular, taking $\left(\begin{smallmatrix}a&b\\c&d \end{smallmatrix}\right)=\left(\begin{smallmatrix} 0 & -1 \\ 1 & 0 \end{smallmatrix}\right)$ and $z = it$ with $t > 0$, we have
\begin{equation}
    \label{eqn:e2_it}
    E_2\left(\frac{i}{t}\right) = -t^2 E_2(it) + \frac{6t}{\pi}.
\end{equation}
It is known that the ring of quasimodular forms of level $1$ is generated by the Eisenstein series $E_2, E_4$ and $E_6$ \cite{kaneko1995generalized}, and the depth of $F$ is equal to the degree of $E_2$ in the polynomial expression of $F$ in terms of $E_2, E_4, E_6$.

Every quasimodular form arises as a linear combination of derivatives of modular forms.
In particular, if $F$ is a quasimodular form of weight $w$ and depth $s$, its derivative $F' := \frac{1}{2\pi i} \frac{\dd F}{\dd z}$ is a quasimodular form of weight $w + 2$ and depth $\le s + 1$.
The derivatives can be computed using the product formula and Ramanujan's formulas:
\begin{equation}
    \label{eqn:ramanujan}
    E_2' = \frac{E_2^2 - E_4}{12}, \quad E_4' = \frac{E_2 E_4 - E_6}{3}, \quad E_6' = \frac{E_2 E_6 - E_4^2}{2}.
\end{equation}

The following identity on $E_2$ will be used later in one of Example \ref{ex:polymod_level}.
\begin{lemma}
    \begin{equation}
    6E_2(2z) = 4 E_2(2z) + E_2\left(\frac{z}{2}\right) + E_2\left(\frac{z+1}{2}\right). \label{eqn:e2_level2}
    \end{equation}
\end{lemma}
\begin{proof}
    This is an Exercise 5.19 of \cite{romik2023topics}.
    For completeness, we provide a proof here. The main idea is to compare the Fourier coefficients of both sides. Comparing coefficients of $q^{n/2}$, \eqref{eqn:e2_level2} is equivalent to
    \[
    6 \sigma_1\left(\frac{n}{2}\right) = 4 \sigma_1\left(\frac{n}{4}\right) + \sigma_1(n) (1 + (-1)^n)
    \]
    where $\sigma_1(a) = 0$ if $a \notin \bZ_{\ge 0}$.
    When $n$ is odd, $\sigma_1(n/2) = \sigma_1(n/4) = 0$ and $1 + (-1)^n = 0$, so the equality holds.
    When $n$ is even, let $n = 2^k m$ where $m$ is odd and $k \ge 1$.
    If $k = 1$, then $\sigma_1(n/2) = \sigma_1(m)$ and $\sigma_1(n/4) = 0$, and the equation reduces to $6 \sigma_1(m) = 2 \sigma_1(2m)$, which follows from the multiplicativity of $\sigma_1$ and $\sigma_1(2) = 3$.
    For $k \ge 2$, again using multiplicativity, we have
    \[
    \sigma_1(n) = \sigma_1(2^k) \sigma_1(m) = (1 + 2 + 2^2 + \cdots + 2^k) \sigma_1(m) = (2^{k+1} - 1) \sigma_1(m)
    \]
    and similar holds for $\sigma_1(n/2)$ and $\sigma_1(n/4)$, and the equation reduces to $6(2^k - 1) = 4(2^{k-1} - 1) + 2(2^{k+1} - 1)$ which is true.
\end{proof}

\subsection{Positive quasimodular forms}
\label{subsec:prelim_posqmf}

One of the core step of Viazovska and Cohn et al.'s proof of the optimality of $E_8$ and Leech lattice packing \cite{viazovska2017sphere,cohn2017sphere} is to show that the \emph{magic functions} satisfy the desired linear constraints, which reduce to inequalities on quasimodular forms.
The original proofs in loc. cit. are based on the interval arithmetic and Sturm's bound, while the subsequent papers by Romik \cite{romik2023viazovska} and the author \cite{lee2024algebraic} provide more conceptual proof of the inequalities.
In particular, the author develops a general theory of \emph{positive} and \emph{completely positive} quasimodular forms in \cite{lee2024algebraic}, where the derivative and Serre derivative play key roles and gives simple and algebraic proofs of the quasimodular form inequalities that arise in the sphere packing papers.
The results in loc. cit., on positivity and (Serre) derivative, are summarized in Proposition \ref{prop:derpos}.

\begin{definition}[L. \cite{lee2024algebraic}]
    A quasimodular form $F$ is \emph{positive} if it takes nonnegative real values on the positive imaginary axis, i.e. $F(it) \ge 0$ for $t > 0$.
    A quasimodular form $F$ is \emph{completely positive} if all the Fourier coefficients of $F$ are real and nonnegative.
    We denote the set of positive (resp. completely positive) quasimodular forms of weight $w$, depth $s$, and level $\Gamma$ by $\QM_{w}^{s, +}(\Gamma)$ (resp. $\QM_{w}^{s, ++}(\Gamma)$).
\end{definition}

\begin{proposition}[L. \cite{lee2024algebraic}]
\label{prop:derpos}
Let $F = \sum_{n \geq n_0} a_n q^n \in \QM_{w}^{s}$ and $F' = \sum_{n \geq n_0} na_n q^n \in \QM_{w+2}^{s + 1}$.
For $k \in \bZ$, let $\partial_k F := F' - \frac{k}{12} E_2 F$ be the Serre derivative of $F$.
\begin{enumerate}
    \item If $F$ is a cusp form, $F \in \QM_{w}^{s, ++}$ if and only if $F' \in \QM_{w+2}^{s +1, ++}$. 
    \item If $F$ is a cusp form, then $F' \in \QM_{w +2}^{s +1, +}$ implies $F \in \QM_{w}^{s, +}$.
    \item $F$ is completely positive if and only if its derivatives are all positive. 
    \item If $\partial_{k} F \in \QM_{w+2}^{s+1, +}$ for some $k$ and all the Fourier coefficients of $F$ are real and the first nonzero Fourier coefficient of $F$ is positive, then $F \in \QM_{w}^{s, +}$.
    \item If $F \in \QM_{w}^{s,++}$ and $n_0 \geq k / 12 \geq 0$, then $\partial_k F \in \QM_{w+2}^{s+1, ++}$.
\end{enumerate}
\end{proposition}

\subsection{Extremal quasimodular forms}
\label{subsec:prelim_extqmf}

Kaneko and Koike \cite{kaneko2006extremal} introduced the notion of \emph{extremal quasimodular forms}, which are quasimodular forms of highest vanishing order at infinity among all quasimodular forms of given weight and depth (of level 1).
Grabner \cite{grabner2020quasimodular} proved certain recurrence relations among extremal quasimodular forms of depth $\le 4$.
In particular, for depth $1$, he proved the following recurrence relation.

\begin{theorem}[Grabner {\cite[Proposition 6.1]{grabner2020quasimodular}}]
    \label{thm:Xw1rec_grabner}
    For $6 \mid w$ and $w \ge 6$, we have
    \begin{align}
        X_{w+2, 1} &= \frac{12}{w+1} \partial_{w-1}X_{w, 1} = \frac{12}{w+1}\left(X_{w, 1}' - \frac{w-1}{12} E_2 X_{w, 1}\right) \label{eqn:recXw2} \\
        X_{w+4, 1} &= E_4 X_{w, 1} \label{eqn:recXw4} \\
        X_{w+6, 1} &= \frac{w+6}{864(w+5)} (E_4 X_{w+2, 1} - E_6 X_{w, 1}). \label{eqn:recXw6}
    \end{align}
\end{theorem}

In \cite{lee2024algebraic}, the author proved another recurrence relation for the extremal quasimodular forms of depth $1$, which gives a proof of the Kaneko and Koike's conjecture.

\begin{theorem}[L. {\cite[Theorem 4.3, Remark 4.5]{lee2024algebraic}}]
    \label{thm:Xw1rec}
    For $6 \mid w$ and $w \ge 12$, we have
    \begin{align}
        X_{w, 1}' &= \frac{5w}{72} X_{6, 1} X_{w-4, 1} + \frac{7w}{72} X_{8, 1} X_{w-6, 1}, \label{eqn:extqmf_rec1} \\
        X_{w+2, 1}' &= \frac{5w}{72} X_{6, 1} X_{w-2, 1} + \frac{7w}{72} X_{8, 1} X_{w-4, 1}, \label{eqn:extqmf_rec2} \\
        X_{w+4, 1}' &= 240 X_{6, 1} X_{w, 1} + \frac{7w}{72} X_{8, 1} X_{w-2, 1} + \frac{5w}{72} X_{10, 1} X_{w-4, 1}. \label{eqn:extqmf_rec3}
    \end{align}
\end{theorem}

Since $X_{w, 1}$ is a depth 1 quasimodular forms, we can write it as
\begin{equation}
    \label{eqn:Xw1AB}
    X_{w, 1} = A_{w} + E_2 B_{w-2}
\end{equation}
where $A_{w}, B_{w-2}$ are modular forms of weight $w$ and $w-2$, respectively.
For the later purpose (Proposition \ref{prop:Xw1monot0}), we will prove a formula for the constant terms of $A_w$ and $B_{w-2}$.
Using Theorem \ref{thm:Xw1rec_grabner}, we can prove the following recurrence relations for $A_w$ and $B_{w-2}$.

\begin{proposition}
    \label{prop:AwBwm2rec}
    For $6 \mid w$ and $w \ge 12$, we have
    \begin{align}
        A_{w+2} &= \frac{12}{w+1} \left(\partial_w A_w - \frac{1}{12} E_4 B_{w-2}\right) \label{eqn:Aw2} \\
        A_{w+4} &= E_4 A_w \label{eqn:Aw4} \\
        A_{w+6} &= \frac{w+6}{864(w+5)} (E_4 A_{w+2} - E_6 A_{w}) \label{eqn:Aw6} \\
        B_{w} &= \frac{12}{w+1} \left(\frac{1}{12} A_w + \partial_{w-2} B_{w-2}\right) \label{eqn:Bw} \\
        B_{w+2} &= E_4 B_{w} \label{eqn:Bw2} \\
        B_{w+4} &= \frac{w+6}{864(w+5)} (E_4 B_{w+2} - E_6 B_{w}). \label{eqn:Bw4}
    \end{align}
\end{proposition}
\begin{proof}
    From \eqref{eqn:recXw2}, we have
    \begin{align}
        X_{w+2, 1} &= A_{w+2} + E_2 B_{w} \label{eqn:recAB1} \\
        &= \frac{12}{w+1} \left(X_{w, 1}' - \frac{w-1}{12} E_2 X_{w, 1}\right) \nonumber \\
        &= \frac{12}{w+1} \left(A_w' + E_2 B_{w-2}' - \frac{w-1}{12} E_2 (A_w + E_2 B_{w-2})\right) \nonumber \\
        &= \frac{12}{w+1} \left(\partial_w A_w + \frac{w}{12} E_2 A_w + E_2 \left(\partial_{w-2} B_{w-2} + \frac{w-2}{12} E_2 B_{w-2}\right) + \frac{E_2^2 - E_4}{12} B_{w-2} \right. \nonumber \\
        &\quad\left.- \frac{w-1}{12} E_2 A_w - \frac{w-1}{12} E_2^2 B_{w-2}\right) \nonumber \\
        &= \frac{12}{w+1} \left(\left(\partial_w A_w - \frac{1}{12} E_4 B_{w-2}\right) + E_2 \left(\partial_{w-2} B_{w-2} + \frac{1}{12} A_w\right)\right) \label{eqn:recAB2}
    \end{align}
    and comparing \eqref{eqn:recAB1} and \eqref{eqn:recAB2} proves \eqref{eqn:Aw2} and \eqref{eqn:Bw}.
    The other relations follow similarly from \eqref{eqn:recXw4} and \eqref{eqn:recXw6}.
\end{proof}

Using these recurrence relations, we can find recurrence relations for the Fourier expansions
\begin{equation}
    A_w = \sum_{n \ge 0} \alpha_{w,n} q^n, \quad B_{w-2} = \sum_{n \ge 0} \beta_{w-2,n} q^n. \label{eqn:AwBwm2fourier}
\end{equation}

\begin{lemma}
    \label{lem:n0coeff}
    For $w \equiv 0 \pmod{6}$, the constant terms $\alpha_{w,0}$ and $\beta_{w-2,0}$ satisfy
    \begin{align}
        \alpha_{w+2,0} &= -\frac{w-1}{w+1} \alpha_{w,0} \label{eqn:alphaw2} \\
        \alpha_{w+4,0} &= \alpha_{w,0} \label{eqn:alphaw4} \\
        \alpha_{w+6,0} &= -\frac{w(w+6)}{432(w+1)(w+5)} \alpha_{w,0} \label{eqn:alphaw6}
    \end{align}
    and $\beta_{w,0} = -\alpha_{w+2,0}$ for all $w \ge 4$.
    In particular, $A_{w}$ and $B_{w-2}$ are not cusp forms for all $w \ge 6$.
    Also, for $w \equiv 0 \pmod{6}$, we have
    \begin{equation}
        \label{eqn:alphawclosed}
        \alpha_{w,0} = (-1)^{\frac{w}{6}}\frac{(w/6)! (w/3)! (w/2)! }{2w \cdot w!}.
    \end{equation}
\end{lemma}
\begin{proof}
    Since $X_{w, 1} = A_w + E_2 B_{w-2} = O(q)$, we have $\beta_{w-2,0} = -\alpha_{w,0}$ by comparing constant terms.
    For $w \equiv 0 \pmod{6}$, by comparing the constant terms in \eqref{eqn:Aw2}, we have
    \[
        \alpha_{w+2,0} = \frac{12}{w+1} \left(-\frac{w}{12} \alpha_{w,0} - \frac{1}{12} \beta_{w-2,0}\right) = -\frac{w-1}{w+1} \alpha_{w,0}.
    \]
    The other two recurrence relations \eqref{eqn:alphaw4} and \eqref{eqn:alphaw6} can be proved similarly using \eqref{eqn:Aw4} and \eqref{eqn:Aw6}.
    From $X_{6, 1} = \frac{E_2 E_4 - E_6}{720}$, we have $\alpha_{6,0} = -\frac{1}{720}$ and $\beta_{4,0} = \frac{1}{720}$, so $\alpha_{w,0}, \beta_{w,0}$ are nonzero and $A_w$, $B_{w-2}$ are not cusp forms.
    Finally, the closed formula for $\alpha_{w,0}$ follows from \eqref{eqn:alphaw6}.
\end{proof}

\begin{lemma}
    \label{lem:n1coeff}
    For $w \ge 12$ with $6 \mid w$, we have
    \begin{align}
        \frac{\alpha_{w, 1}}{\alpha_{w, 0}} &= -\frac{12(w-3)(w+4)}{w-6} \\
        \frac{\alpha_{w+2,1}}{\alpha_{w+2,0}} &= - \frac{12(w^2 - 9w - 24)}{w-6} \\
        \frac{\alpha_{w+4,1}}{\alpha_{w+4,0}} &= - \frac{12(w^2 - 19w + 108)}{w-6} \\
        \frac{\beta_{w-2,1}}{\beta_{w-2,0}} &= -\frac{12(w-1)w}{w-6} \\
        \frac{\beta_{w,1}}{\beta_{w,0}} &= -\frac{12(w-12)(w+1)}{w-6} \\
        \frac{\beta_{w+2,1}}{\beta_{w+2,0}} &= - \frac{12(w^2 - 21w + 120)}{w-6} 
    \end{align}
\end{lemma}
\begin{proof}
    Use induction on $w$.
    The base case $w = 12$ can be checked directly from
    \begin{align*}
        X_{12,1} &= \frac{-12 E_2 E_4 E_6 + 5E_4^3 + 7E_6^2}{3991680} = A_{12} + E_2 B_{10}, \\
        A_{12} &= \frac{5 E_4^3 + 7 E_6^2}{3991680} = \frac{1}{332640} - \frac{1}{1155}q + O(q^2), \quad B_{10} = -\frac{E_4 E_6}{332640} = -\frac{1}{332640} + \frac{1}{1260} q + O(q^2), \\
        X_{14, 1} &= \frac{7 E_2 E_4^3 + 5 E_2 E_6^2 - 12 E_4^2 E_6}{4717440} = A_{14} + E_2 B_{12}, \\
        A_{14} &= -\frac{E_4^2 E_6}{391320} = - \frac{1}{391320} + \frac{1}{16380} q + O(q^2), \quad B_{12} = \frac{7 E_4^3 + 5 E_6^2}{4717440} = \frac{1}{393120} + O(q^2), \\
        X_{16, 1} &= \frac{E_4 (-12 E_2 E_4 E_6 + 5 E_4^2 + 7 E_6^2)}{3991680} = A_{16} + E_2 B_{14}, \\
        A_{16} &= \frac{E_4 (5 E_4^2 + 7 E_6^2)}{3991680} = \frac{1}{332640} - \frac{1}{6930} q + O(q^2), \quad B_{14} = -\frac{E_4^2 E_6}{332640} = -\frac{1}{332640} + \frac{1}{13860} q + O(q^2).
    \end{align*}
    By comparing the coefficients of $q$ in \eqref{eqn:Aw2}, we have
    \begin{align*}
        \alpha_{w+2,1} &= \frac{12}{w+1} \left(2(w+10) \alpha_{w,0} + \left(-\frac{w}{12}+1\right)\alpha_{w,1} - \frac{1}{12} \beta_{w-2,1}\right)
    \end{align*}
    and thus
    \begin{align*}
        \frac{\alpha_{w+2,1}}{\alpha_{w+2,0}} &= \frac{12}{w+1} \left(2(w+10) \cdot \frac{\alpha_{w,0}}{\alpha_{w+2,0}} + \left(-\frac{w}{12}+1\right) \cdot \frac{\alpha_{w,0}}{\alpha_{w+2,0}} \frac{\alpha_{w,1}}{\alpha_{w,0}} + \frac{1}{12} \cdot \frac{\alpha_{w,0}}{\alpha_{w+2,0}} \frac{\beta_{w-2,1}}{\beta_{w-2,0}}\right) \\
        &= -\frac{12}{w-1} \left(2(w+10) + \left(-\frac{w}{12} + 1\right) \cdot \frac{-12(w-3)(w+4)}{w-6} + \frac{1}{12} \cdot -\frac{12(w-1)w}{w-6}\right) \\
        &= - \frac{12(w^2 - 9w - 24)}{w-6}.
    \end{align*}
    The other formulas can be proved similarly.
\end{proof}

\section{Monotonicity of $t^m F(it)$}
\label{sec:mono}

In this section, we provide several methods to prove the monotonicity of the function \eqref{eqn:tmf}, via Lambert series expansion (Section \ref{subsec:mono_lambert}) or reducing it to a quasimodular form inequality (Section \ref{subsec:mono_qmfineq}).

\subsection{Lambert series approach}
\label{subsec:mono_lambert}

For certain low-weight quasimodular forms $F(z)$, it is possible to prove monotonicity of $t^m F(it)$ via \emph{Lambert series}, i.e. expansions of the form
\[
F(z) = \sum_{n \ge 1} b_n \frac{q^n}{(1 - q^n)^k}.
\]
Once we know that $b_n \ge 0$ for all $n$, monotonicity of $t^m F(it)$ reduces to that of
\[
t \mapsto t^m \frac{e^{-t}}{(1 - e^{-t})^k},
\]
which are often easy to check.
For example, we can prove monotonicity of the following functions.

\begin{lemma}
    \label{lem:polymod_E2}
    The function $t \mapsto t^2 (1 - E_2(it))$ is monotone decreasing for $t > 0$.
\end{lemma}
\begin{proof}
    From \eqref{eqn:e2fourier}, we have
    \[
    1 - E_2(z) = 24 \sum_{n \ge 1} \sigma_1(n) q^n = 24 \sum_{m, d \ge 1} d q^{dm} = 24 \sum_{m \ge 1} \frac{q^m}{(1 - q^m)^2},
    \]
    so it is enough to check that the following function is monotone in $t$:
    \[
        g_2(t) := t^2 \frac{e^{-t}}{(1 - e^{-t})^2}.
    \]
    This follows from
    \[
    \frac{\dd}{\dd t} g_2(t) = -\frac{e^t t (e^t (t - 2) + (t + 2))}{(e^t - 1)^3}
    \]
    and
    \[
    e^t(t - 2) + (t + 2) > 0 \Leftrightarrow \tanh\left(\frac{t}{2}\right) < \frac{t}{2},
    \]
    where the last inequality can be checked by considering $h(t) = t - \tanh(t)$, so that $h'(t) = \tanh^2(t) > 0$ and $h(0) = 0$.
\end{proof}

\begin{lemma}
    \label{lem:polymod_X42}
    The function $t \mapsto t^3 X_{4, 2}(it)$ is monotone decreasing for $t > 0$.
\end{lemma}
\begin{proof}
    We can write $X_{4, 2}$ as
    \begin{equation*}
        X_{4, 2}(z) = \sum_{n \ge 1} n \sigma_1(n) q^n = \sum_{m \ge 1} \sum_{d \ge 1} d^{2} m q^{dm} = \sum_{m \ge 1} \frac{m q^m (1 + q^m)}{(1 - q^m)^3},
    \end{equation*}
    where the last equality follows from $\sum_{k \ge 1} k^2 x^k = \frac{x(1+x)}{(1 - x)^3} $.
    Hence the monotonicity of $t^3 X_{4, 2}(it)$ reduces to that of
    \begin{equation*}
        g_3(t) := t^3 \frac{e^{-t}(1 + e^{-t})}{(1 - e^{-t})^3}.
    \end{equation*}
    The derivative of $g_3(t)$ is
    \begin{equation*}
        \frac{\dd}{\dd t} g_3(t) = -\frac{e^t t^2 (t(e^{2t} + 4 e^t + 1) - 3(e^{2t} - 1))}{(e^t - 1)^4} = - \frac{3t^2 e^t (e^{2t} + 4e^t + 1)}{(e^t - 1)^4} \left(\frac{t}{3} - \frac{e^{2t} - 1}{e^{2t} + 4e^t + 1}\right).
    \end{equation*}
    If we put
    \[
    h_3(t) := \frac{t}{3} - \frac{e^{2t} - 1}{e^{2t} + 4e^t + 1},
    \]
    then $\dd g_3 / \dd t < 0$ if and only if $h_3(t) > 0$, which follows from $h_3(0) = 0$ and
    \begin{equation*}
        \frac{\dd}{\dd t} h_3(t) = \frac{(e^t - 1)^4}{3(e^{2t} + 4e^t + 1)^2} > 0.
    \end{equation*}
\end{proof}

\begin{lemma}
    \label{lem:polymod_D2}
    The function $t \mapsto t^2 (E_2(2it) - E_2(it))$ is monotone decreasing for $t > 0$.
\end{lemma}
\begin{proof}
    We can write $E_2(2z) - E_2(z)$ as
    \begin{equation*}
        E_2(2z) - E_2(z) = 24 \sum_{n \ge 1} \sigma_1(n) q^n - 24 \sum_{n \ge 1} \sigma_1(n) q^{2n} = 24\sum_{n \ge 1} \left(\frac{q^n}{(1 - q^n)^2} - \frac{q^{2n}}{(1 - q^{2n})^2}\right),
    \end{equation*}
    so it is enough to check that the following function is monotone in $t$:
    \[
        g_{2,2}(t) := t^2 \left(\frac{e^{-t}}{(1 - e^{-t})^2} - \frac{e^{-2t}}{(1 - e^{-2t})^2}\right) = t^2 \frac{e^{-t}(1 + e^{-t} + e^{-2t})}{(1 - e^{-2t})^2}.
    \]
    The derivative of $g_{2,2}(t)$ is
    \[
    \frac{\dd}{\dd t} g_{2, 2}(t) = - \frac{2t e^t (e^{4t} + 2e^{3t} + 6e^{2t} + 2e^t + 1)}{(e^{2t} - 1)^3} \left[\frac{t}{2} - \frac{(e^t + 1)(e^{3t} - 1)}{(e^{4t} + 2e^{3t} + 6e^{2t} + 2e^{t} + 1)}\right].
    \]
    If we put
    \[
    h_{2,2}(t) := \frac{t}{2} - \frac{(e^t + 1)(e^{3t} - 1)}{(e^{4t} + 2e^{3t} + 6e^{2t} + 2e^{t} + 1)},
    \]
    then $\dd g_{2,2} / \dd t < 0$ if and only if $h_{2,2}(t) > 0$, which follows from $h_{2,2}(0) = 0$ and
    \[
    \frac{\dd}{\dd t} h_{2,2}(t) = \frac{(e^t - 1)^4 (e^t + 1)^4 (e^{2t} + 4e^{t} + 1)}{2(e^{4t} + 2e^{3t} + 6e^{2t} + 2e^{t} + 1)^2} > 0.
    \]
\end{proof}

\begin{lemma}
    \label{lem:X81weak}
    The function $t \mapsto t^{6} X_{8, 1}(it)$ is monotone decreasing for $t > 0$.
\end{lemma}
\begin{proof}
    We can write $X_{8, 1}$ as
    \[
        X_{8, 1}(z) = \sum_{n \ge 1} n \sigma_5(n) q^n = \sum_{m \ge 1} \sum_{d \ge 1} d^6 m q^{dm} = \sum_{m \ge 1} \frac{m q^m (q^{5m} + 57 q^{4m} + 302 q^{3m} + 302 q^{2m} + 57 q^m + 1)}{(1 - q^m)^7}
    \]
    so it is enough to check that the following function is monotone in $t$:
    \[
    g_6(t) := t^6 \frac{e^{-t} (e^{-5t} + 57 e^{-4t} + 302 e^{-3t} + 302 e^{-2t} + 57 e^{-t} + 1)}{(1 - e^{-t})^7}.
    \]
    This follows from
    \[
    \frac{\dd}{\dd t} g_6(t) = -\frac{e^t t^5 (t P(e^t) - Q(e^t))}{(e^t - 1)^8}
    \]
    where $P(x)$ and $Q(x)$ are polynomials
    \[
    P(x) = x^6 + 120 x^5 + 1191 x^4 + 2416 x^3 + 1191 x^2 + 120 x + 1,\,\, Q(x) = 6 (x^6 + 56 x^5 + 245 x^4 - 245 x^2 - 56 x - 1)
    \]
    and $t P(e^t) - Q(e^t) > 0$ for $t > 0$ follows from
    \[
    \frac{\dd}{\dd t} \left[t - \frac{Q(e^t)}{P(e^t)}\right] = \frac{R(e^t)}{P(e^t)^2} > 0,
    \]
    where
    \begin{align*}
    R(x) &= (x^6 - 72 x^5)^2 + (1 - 72 x)^2 + 246 (x^{10} + x^2) + 23408 (x^9 + x^3) \\
    &\quad + 342687 (x^8 + x^4) + 754464 (x^7 + x^5) + 1377108 x^6
    \end{align*}
    which is positive for all $x > 1$, and $Q(1) = 0$.
\end{proof}

\begin{lemma}
    \label{lem:X101weak}
    The function $t \mapsto t^{8} X_{10, 1}(it)$ is monotone decreasing for $t > 0$.
\end{lemma}
\begin{proof}
    We can write $X_{10, 1}$ as
    \begin{align*}
        X_{10, 1}(z) &= \sum_{n \ge 1} n \sigma_7(n) q^n = \sum_{m \ge 1} \sum_{d \ge 1} d^8 m q^{dm} \\
        &= \sum_{m \ge 1} \frac{m q^m ((q^{7m} + 1) + 247 (q^{6m} + q^{m}) + 4293 (q^{5m} + q^{2m}) + 15619 (q^{4m} + q^{3m}))}{(1 - q^m)^9}
    \end{align*}
    so it is enough to check that the following function is monotone in $t$:
    \[
    g_8(t) := t^8 \frac{e^{-t} ((e^{-7t} + 1) + 247 (e^{-6t} + e^{-t}) + 4293 (e^{-5t} + e^{-2t}) + 15619 (e^{-4t} + e^{-3t}))}{(1 - e^{-t})^9}.
    \]
    Its derivative is of the form
    \[
    -\frac{e^{-9t} t^7 (t P(e^t) - Q(e^t))}{(1 - e^{-t})^{10}}
    \]
    where $P(x)$ and $Q(x)$ are polynomials
    \begin{align*}
        P(x) &= (x^8 + 1) + 502 (x^7 + x) + 14608 (x^6 + x^2) + 88234 (x^5 + x^3) + 156190 x^4, \\
        Q(x) &= 8 (x^8 - 1 + 246 (x^7 - x) + 4046 (x^6 - x^2) + 11326 (x^5 - x^3)),
    \end{align*}
    and it is enough to show that
    \begin{equation}
        \label{eqn:fpos}
        f(t) := t P(e^t) - Q(e^t) > 0 \text{ for } t > 0.
    \end{equation}
    To prove this, we consider its Taylor expansion at $t = 0$:
    \[
    f(t) = \sum_{k=0}^{8} (a_k t + b_k) e^{kt} = \sum_{k = 0}^{8} \left(\sum_{n \ge 1} \frac{(na_k + kb_k)k^{n-1}}{n!} t^n + b_k\right)  = \sum_{n \ge 0} \frac{c_n}{n!} t^n
    \]
    where $P(x) = \sum_{0 \le k \le 8} a_k x^k$, $Q(x) = \sum_{0 \le k \le 8} b_k x^k$, $c_0 = f(0) = 8 > 0$ and
    \begin{align*}
    c_n &= \sum_{k=1}^{8} (na_k + kb_k) k^{n-1} \\
    &= (n - 64) 8^{n-1} + (502n - 13376) 7^{n-1} + (14608 n - 194208) 6^{n-1} \\
    &\quad + (88234 n - 453040) 5^{n-1} + 156190 n 4^{n-1} + (88234n + 271824) 3^{n-1}  \\
    &\quad + (14608 n + 64736) 2^{n-1} + (502 n + 1968).
    \end{align*}
    So it is enough to show that $c_n > 0$ for all $n \ge 1$.
    For $n \ge 65$, all the linear coefficients are positive, so $c_n > 0$.
    We can check $c_n > 0$ for $1 \le n \le 64$ directly.
\end{proof}

\begin{remark}
Figure \ref{fig:g8g9} shows the graphs of $g_8(t)$ and $g_9(t)$ on $2.5 \le t \le 20$, where $g_9(t)$ is defined as the function we obtain by replacing $t^8$ in $g_8(t)$ with $t^9$.
The function $g_9(t)$ is not monotone decreasing, which can be checked rigorously by comparing the limit of $g_9(t)$ as $t \to 0^+$ and the value $g_9(10)$.
In fact, the function $t \mapsto t^9 X_{10, 1}(it)$ is not monotone decreasing; see Remark \ref{remark:X81X101}.
\end{remark}

\begin{figure}[h]
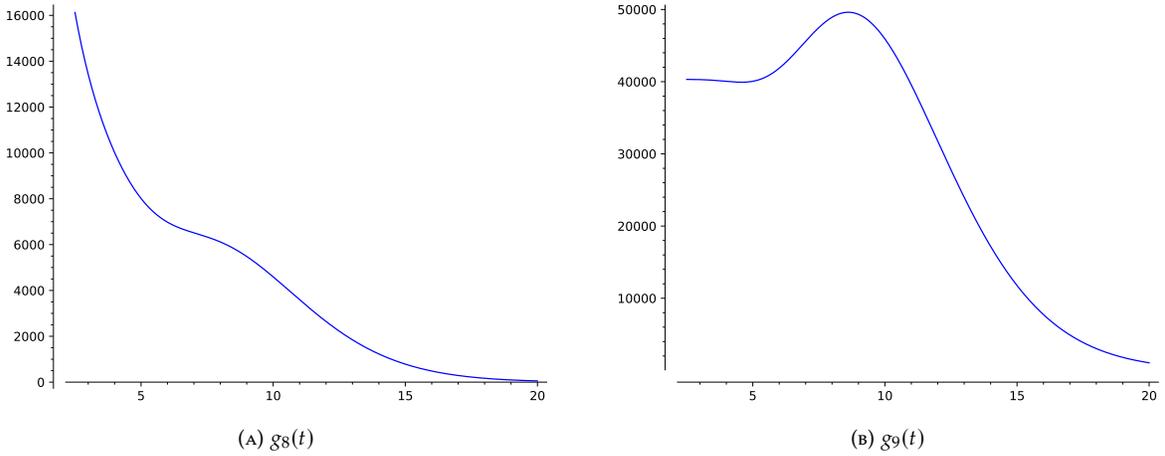
%
    \centering
    \subfloat[\centering $g_8(t)$]{{\includesvg[width=0.45\textwidth]{images/t8.svg}}}%
    \qquad
    \subfloat[\centering $g_9(t)$]{{\includesvg[width=0.45\textwidth]{images/t9.svg}}}%
    \caption{Graphs of $g_8(t)$ and $g_9(t)$ on $2.5 \le t \le 20$.}%
    \label{fig:g8g9}%
\end{figure}

Unfortunately, such an approach does not always work.
For example, the function
\[
t \mapsto t^4 (E_4(it) - 1)
\]
is monotone decreasing for $t > 0$ (see Corollary \ref{cor:polymod_E4}).
However, the following Lambert series expansion
\[
E_4(z) - 1 = 240 \sum_{n \ge 1} \sigma_3(n)q^n = 240 \sum_{m, d \ge 1} d^2 q^{dm} = 240 \sum_{m \ge 1} \frac{q^m(1 + 4q^m + q^{2m})}{(1 - q^m)^4}
\]
is not enough to prove monotonicity, since the function
\[
t \mapsto t^4 \frac{e^{-t}(1 + 4e^{-t} + e^{-2t})}{(1 - e^{-t})^4}
\]
is not monotone decreasing; it is $6 + \frac{t^4}{120} + O(t^5)$ near $t = 0$.
Similarly, the function
\[
t \mapsto t^5 X_{6, 1}(it)
\]
is monotone decreasing for $t > 0$ (Corollary \ref{cor:polymod_X61}).
However, the Lambert series expansion of $X_{6, 1}$ is 
\begin{align*}
    X_{6, 1}(z) = \sum_{m \ge 1} \sum_{d \ge 1} d^4 m q^{dm} = \sum_{m \ge 1} \frac{mq^m (1 + 11q^m + 11 q^{2m} + q^{3m})}{(1 - q^m)^5}
\end{align*}
but the function
\[
t \mapsto t^5 \frac{e^{-t} (1 + 11 e^{-t} + 11 e^{-2t} + e^{-3t})}{(1 - e^{-t})^5}
\]
is not monotone decreasing; it is $24 + \frac{t^6}{252} + O(t^7)$ near $t = 0$.

Another non-example is the weight $12$ and depth $1$ extremal quasimodular form
\[
X_{12, 1} = - \frac{E_{10}'}{277200} - \frac{\Delta}{1050} = \frac{1}{1050} \sum_{n \ge 1} (n \sigma_9(n) - \tau(n)) q^n
\]
which does not admit a nice Lambert series expansion, because of the presence of the term $\tau(n)$.
Monotonicity of relevant functions will be proved in the following section with different approach (Corollary \ref{cor:polymod_E4} and Corollary \ref{cor:polymod_X121}).

\subsection{Two sufficient conditions and more examples}
\label{subsec:mono_qmfineq}

We introduce two simple sufficient conditions (Proposition \ref{prop:polymod_tangent} and \ref{prop:polymod_der}) which are useful to prove monotonicity of $t^m F(it)$.

\begin{proposition}
    \label{prop:polymod_tangent}
    Let $F$ be a quasimodular form and let $m \in \bR_{> 0}$.
    Assume that
    \begin{enumerate}
        \item $F$ and $F'$ are positive quasimodular forms,
        \item
        \[
        \lim_{t \to 0^+} \frac{F(it)}{t F'(it)} = \frac{2\pi}{m},
        \]
        \item $(m+1) (F')^2 - m F'' F$ is positive.
    \end{enumerate}
    Then the function $t \mapsto t^m F(it)$ is monotone decreasing on $(0, \infty)$.
    Also, $m$ is the largest possible constant such that $t^m F(it)$ is monotone decreasing.
\end{proposition}
\begin{proof}
    Consider the function
    \[
    h(t) := \frac{F(it)}{F'(it)}.
    \]
    The limit condition implies that $h(0) := \lim_{t \to 0^+} h(t) = 0$ and $h'(0) = \frac{2\pi}{m}$.
    In other words, the tangent line of $h(t)$ at $t = 0$ is $l(t) = \frac{2\pi}{m} t$.
    Direct computation shows that
    \[
    \frac{\dd}{\dd t}(t^m F(it)) = m t^{m-1} F(it) + t^m (-2\pi) F'(it) = m t^{m-1} F'(it) \left(h(t) - l(t)\right),
    \]
    so it is enough to show that $h(t) < l(t)$ for all $t > 0$.
    Consider their difference $g(t) = l(t) - h(t)$.
    The limit condition implies that $g(0) = 0$ and it is enough to show that $g(t)$ is monotone increasing, and this is equivalent to
    \[
    \frac{\dd}{\dd t} g(t) = \frac{2\pi}{m} + 2\pi \cdot \frac{(F'(it))^2 - F''(it) F(it)}{(F'(it))^2} > 0 \Leftrightarrow (m+1) (F')^2 - m F'' F > 0.
    \]
    The last claim follows from considering the behavior of $h(t)$ near $t = 0$.
\end{proof}

\begin{remark}
    The modular form $(m+1) (F')^2 - m F'' F$ is a constant multiple of the Rankin--Cohen bracket \cite{rankin1956construction,cohen1975sums} of $F$ with itself.
    More precisely, Martin and Royer \cite{martin2009rankin} generalized the notion of Rankin--Cohen brackets for quasimodular forms as
    \[
    \Phi_{n;k,s;\ell,t}(F, G) = \sum_{r=0}^{n} (-1)^r \binom{k-s+n-1}{n-r} \binom{\ell-t+n-1}{r} (D^r F) (D^{n-r} G)
    \]
    for $0 \le s \le k/2$ and $0 \le t \le \ell/2$.
    This matches with the usual Rankin--Cohen bracket $[F, G]_n^{(k, \ell)}$ when $s = t = 0$.
    One can directly check that
    \[
    \Phi_{2;m,0;m,0}(F, F) = [F, F]_{2}^{(m,m)} = \sum_{r=0}^{2} (-1)^r \binom{m+1}{2-r} \binom{m+1}{r} (D^r F)(D^{2-r}F) = (m+1)m F''F - (m+1)^2 (F')^2
    \]
    and the 3rd condition of Proposition \ref{prop:polymod_tangent} is equivalent to the negativity of $\Phi_{2;m,0;m,0}(F, F)$.
    In particular, when $F$ has weight $w$ and depth $s$, then for $m = w - s$
    \[
    (m+1) (F')^2 - m F'' F = - \frac{1}{m + 1} \Phi_{2;w,s;w,s}(F, F).
    \]
    is a quasimodular form of weight $2w + 4$ and depth at most $2s$.
\end{remark}

\begin{remark}
    For any completely positive quasimodular form $F$, we have $F'' F - (F')^2 > 0$.
    In fact, this holds for any completely monotone functions (e.g. see \cite{fink1982kolmogorov}).
\end{remark}

As corollaries, we can prove monotonicity of several functions including extremal quasimodular forms, such as the following example (which cannot be proved by the Lambert series approach in Section \ref{subsec:mono_lambert}).

\begin{proposition}
    \label{prop:Xw1limit}
    Let $X_{w, 1}$ be the normalized extremal quasimodular form of weight $w$ and depth $1$.
    Then
    \begin{align}
        \lim_{t \to 0^+} \frac{X_{w, 1}(it)}{t X_{w, 1}'(it)} &= \frac{2\pi}{w-1} \label{eqn:extremallimit} \\
        \lim_{t \to 0^+} t^{w-1} X_{w, 1}(it) &= -\frac{6(-1)^{\frac{w}{2}}\beta_{w-2}}{\pi} \label{eqn:extremallimit2}
    \end{align}
    for all $w \ge 6$.
\end{proposition}
\begin{proof}
    The transformation law gives
    \[
        X_{w, 1}\left(-\frac{1}{z}\right) = z^w A_w(z) + \left(z^2 E_2(z) - \frac{6iz}{\pi}\right) z^{w-2}B_{w-2}(z) = z^w X_{w, 1}(z) - \frac{6iz^{w-1}}{\pi} B_{w-2}(z)
    \]
    and $z = it$ gives
    \begin{equation}
        X_{w,1} \left(\frac{i}{t}\right) = (-1)^{\frac{w}{2}}\left[t^w X_{w, 1}(it) - \frac{6 t^{w-1}}{\pi} B_{w-2}(it)\right]. \label{eqn:Xw1trans}
    \end{equation}
    By differentiating $X_{w, 1}$ and rewriting in terms of Serre derivatives, we have
    \begin{align*}
        X_{w, 1}' &= A_{w}' + E_2' B_{w-2} + E_2 B_{w-2}' \\
        &= \partial_w A_w + \frac{w}{12} E_2 A_{w} + \frac{E_2^2 - E_4}{12} B_{w-2} + E_2 \left(\partial_{w-2} B_{w-2} + \frac{w-2}{12} E_2 B_{w-2}\right) \\
        &= \left(\partial_{w} A_{w} - \frac{1}{12} E_4 B_{w-2}\right) + E_2 \left(\frac{w}{12} A_w + \partial_{w-2} B_{w-2}\right) + E_2^2 \cdot \frac{w-1}{12} B_{w-2} \\
        &=: \widetilde{A}_{w+2} + E_2 \widetilde{B}_{w} + E_2^2 \widetilde{C}_{w-2},
    \end{align*}
    where
    \begin{align*}
        \widetilde{A}_{w+2} &:= \partial_{w} A_{w} - \frac{1}{12} E_4 B_{w-2} \in \cM_{w+2}(\SL_2(\bZ)), \\
        \widetilde{B}_{w} &:= \frac{w}{12} A_w + \partial_{w-2} B_{w-2} \in \cM_{w}(\SL_2(\bZ)), \\
        \widetilde{C}_{w-2} &:= \frac{w-1}{12} B_{w-2} \in \cM_{w-2}(\SL_2(\bZ)).
    \end{align*}
    Then applying the transformation law with $z = it$ gives
    \begin{align}
        X_{w, 1}'\left(\frac{i}{t}\right) &= (-1)^{\frac{w}{2} + 1} \left[t^{w+2} X_{w, 1}'(it) + \frac{t^{w+1}}{\pi} (-6 \widetilde{B}_{w}(it) - (w-1) B_{w-2}(it) E_{2}(it)) + \frac{t^w}{\pi^2} \cdot 3(w-1) B_{w-2}(it)\right] \nonumber\\
        &= (-1)^{\frac{w}{2} + 1} \left[t^{w+2} X_{w, 1}'(it) + \frac{t^{w+1}}{\pi} \left(-\frac{w}{2} X_{w, 1}(it) - 6 B_{w}'(it)\right) + \frac{t^w}{\pi^2} \cdot 3(w-1) B_{w-2}(it)\right]  \label{eqn:dXw1trans}
    \end{align}
    By combining \eqref{eqn:Xw1trans} and \eqref{eqn:dXw1trans}, we have
    \begin{align*}
        \lim_{t \to 0^+} \frac{X_{w, 1}(it)}{t X_{w, 1}'(it)} &= \lim_{t \to \infty} \frac{X_{w, 1}(i/t)}{\frac{1}{t} X_{w, 1}'(i/t)} \\
        &= \lim_{t \to \infty} -\frac{t^w X_{w, 1}(it) - \frac{6 t^{w-1}}{\pi} B_{w-2}(it)}{t^{w+1} X_{w, 1}'(it) + \frac{t^w}{\pi} \left(-\frac{w}{2} X_{w, 1}(it) - 6 B_{w-2}'(it)\right) + \frac{t^{w-1}}{\pi^2} \cdot 3(w-1) B_{w-2}(it)} \\
        &= -\lim_{t \to \infty} \frac{t X_{w, 1}(it) - \frac{6}{\pi} B_{w-2}(it)}{t^2 X_{w, 1}'(it) + \frac{t}{\pi} \left(-\frac{w}{2} X_{w, 1}(it) - 6 B_{w-2}'(it)\right) + \frac{1}{\pi^2} \cdot 3(w-1) B_{w-2}(it)}.
    \end{align*}
    Since $X_{w, 1}, X_{w, 1}'$ and $-\frac{w}{2} X_{w, 1} - 6 B_{w-2}'$ are cusp forms, the limits of the terms involving them vanish as $t \to \infty$.
    By Lemma \ref{lem:n0coeff}, $B_{w-2}$ is not a cusp form and the final limit equals to
    \[
    -\frac{-\frac{6}{\pi} \cdot \beta_{w-2}}{\frac{1}{\pi^2} \cdot 3(w-1) \cdot \beta_{w-2}} = \frac{2\pi}{w-1}
    \]
    which proves \eqref{eqn:extremallimit}.
    The limit \eqref{eqn:extremallimit2} follows from \eqref{eqn:Xw1trans} and the cuspidality of $X_{w, 1}$.
\end{proof}

\begin{corollary}
    \label{cor:polymod_X61}
    $t \mapsto t^5 X_{6, 1}(it)$ is monotone decreasing for $t > 0$.
\end{corollary}
\begin{proof}
    We know that $X_{6, 1}$ is completely positive.
    The limit condition follows from Proposition \ref{prop:Xw1limit} with $w = 6$.
    Direct computation shows that
    \[
    6 (X_{6, 1}')^2 - 5 X_{6, 1}'' X_{6, 1} = \Delta X_{4, 2} > 0.
    \]
\end{proof}

\begin{corollary}
    \label{cor:polymod_E4}    
    $t \mapsto t^4 (E_4(it) - 1)$ is monotone decreasing for $t > 0$.
\end{corollary}
\begin{proof}
    The monotonicity is equivalent to
    \[
    \frac{\dd}{\dd t} \left[t^4 (E_4(it) - 1)\right] = 4 t^3 (E_4(it) - 1) - 2\pi t^4 E_4'(it) < 0 \Leftrightarrow E_4(it) - 1 - \frac{\pi t}{6}(E_2(it) E_4(it) - E_6(it)) < 0.
    \]
    Substituting $t$ with $1/t$ and using the transformation laws, this is equivalent to
    \begin{align*}
        E_4\left(\frac{i}{t}\right) - 1 - \frac{\pi}{6t}\left[E_2\left(\frac{i}{t}\right) E_4\left(\frac{i}{t}\right) - E_6\left(\frac{i}{t}\right)\right] &= t^4 E_4(it) - 1 - \frac{\pi}{6t}\left[\left(-t^2 E_2(it) + \frac{6t}{\pi}\right) \cdot t^4 E_4(it) + t^6 E_6(it)\right] \\
        &= -1 + \frac{\pi t^5}{6}(E_2(it) E_4(it) - E_6(it)) \\
        &= -1 + 120 \pi t^5 X_{6, 1}(it) < 0.
    \end{align*}
    The last inequality follows from Corollary \ref{cor:polymod_X61} and $\lim_{t \to 0^+} t^5 X_{6, 1}(it) = \frac{1}{120 \pi}$ by Proposition \ref{prop:Xw1limit}.
\end{proof}

\begin{corollary}
    \label{cor:polymod_X121}
    $t \mapsto t^{11}X_{12, 1}(it)$ is monotone decreasing for $t > 0$.
\end{corollary}
\begin{proof}
    We know that $X_{12, 1}$ is completely positive.
    The limit condition follows from Proposition \ref{prop:Xw1limit} with $w = 12$.
    Direct computation shows that
    \[
        12 (X_{12, 1}')^2 - 11 X_{12, 1}'' X_{12, 1} = \frac{1}{2^{10} \cdot 3^{6} \cdot 5^{2} \cdot 7^{2}}\cdot \Delta F,
    \]
    where
    \[
    F = 49 E_2^2 E_4^3 - 25 E_2^2 E_6^2 - 48 E_2 E_4^2 E_6 - 25 E_4^4 + 49 E_4 E_6^2
    \]
    is the quasimodular form of weight $16$ and depth $2$ that appears in the Leech lattice packing \eqref{eqn:F}, which is positive \cite[Corollary 6.2]{lee2024algebraic}.
\end{proof}

\begin{corollary}
    \label{cor:polymod_X141}
    $t \mapsto t^{13}X_{14, 1}(it)$ is monotone decreasing for $t > 0$.
\end{corollary}
\begin{proof}
    Direct computation shows that
    \[
    14 (X_{14, 1}')^2 - 13 X_{14, 1}'' X_{14, 1} = 4 \Delta^2 X_{8, 2},
    \]
    where $X_{8, 2}$ is the normalized extremal quasimodular form of weight $8$ and depth $2$, which is completely positive by \cite[Proposition 4.6]{lee2024algebraic}.
\end{proof}

\begin{remark}
\label{remark:X81X101}
Note that $t \mapsto t^{7} X_{8, 1}(it)$ and $t \mapsto t^{9} X_{10, 1}(it)$ are not monotone decreasing (See Figure \ref{fig:X81} and \ref{fig:X101}).
The former function has a local maximum at $t = 1$; we have
\[
\frac{\dd}{\dd t} (t^{7} X_{8, 1}(it)) = t^6 (7 X_{8, 1}(it) - 2\pi t X_{8, 1}'(it))
\]
and from $E_2(i) = \frac{3}{\pi}$ and $E_6(i) = 0$,
\begin{align*}
    7 X_{8, 1}(i) - 2\pi X_{8, 1}'(i) &= 7 \cdot \frac{E_4(i)^2 - E_2(i) E_6(i)}{1008} - 2\pi \left(\frac{-E_2(i)^2 E_6(i) + 2 E_2(i) E_4(i)^2 - E_4(i) E_6(i)}{1728}\right) \\
    &= E_4(i)^2 \left(\frac{7}{1008} - \frac{2\pi \cdot 2 \cdot \frac{3}{\pi}}{1728}\right) = 0.
\end{align*}
For the latter function, note that
\[
X_{10, 1} = \frac{E_4(E_2 E_4 - E_6)}{720}, \quad X_{10, 1}' = \frac{9 E_2^2 E_4^2 - 18 E_2 E_4 E_6 + 5 E_4^3 + 4 E_6^2}{8640}.
\]
By Corollary \ref{prop:Xw1limit}, the limit as $t \to 0^+$ is $-\frac{6(-1)^{\frac{10}{2}} \beta_8}{\pi} = \frac{1}{120\pi}$, while the value at $t = 1$ is
\[
X_{10, 1}(i) = \frac{E_2(i) E_4(i)^2}{720} = \frac{1}{720} \cdot \frac{3}{\pi} \cdot \frac{9 \Gamma(1/4)^{16}}{4096 \pi^{12}} > \frac{1}{120\pi}.
\]
\end{remark}

\begin{figure}[t]
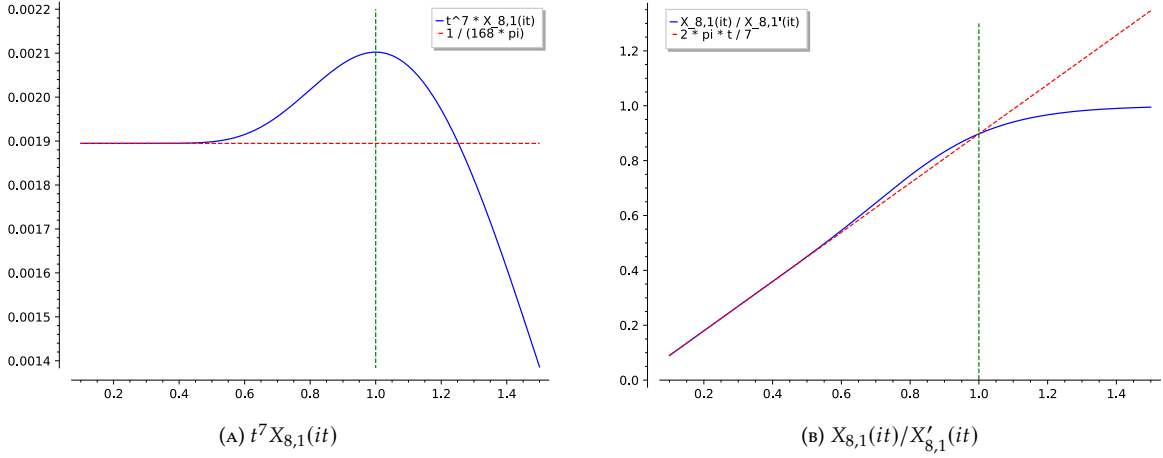
%
    \centering
    \subfloat[\centering $t^{7}X_{8, 1}(it)$]{{\includesvg[width=0.45\textwidth]{images/X81_1.svg}}}%
    \qquad
    \subfloat[\centering $X_{8, 1}(it) / X_{8, 1}'(it)$]{{\includesvg[width=0.45\textwidth]{images/X81_2.svg}}}%
    \caption{Graphs of $t \mapsto t^{7} X_{8, 1}(it)$ and $t \mapsto X_{8, 1}(it) / X_{8, 1}'(it)$ with asymptotes as $t \to 0^+$.}%
    \label{fig:X81}%
\end{figure}

\begin{figure}[t]
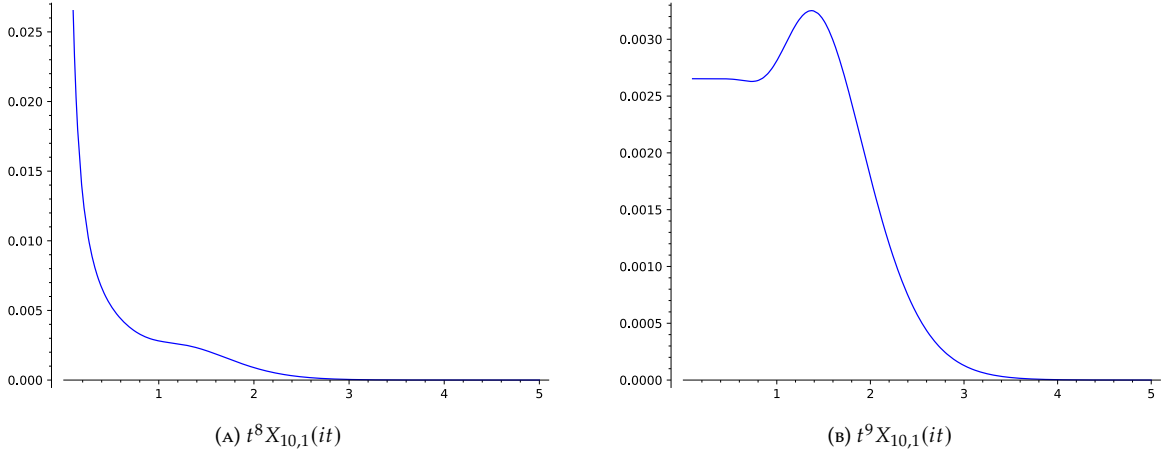
%
    \centering
    \subfloat[\centering $t^8 X_{10, 1}(it)$]{{\includesvg[width=0.45\textwidth]{images/t8X101.svg}}}%
    \qquad
    \subfloat[\centering $t^9 X_{10, 1}(it)$]{{\includesvg[width=0.45\textwidth]{images/t9X101.svg}}}%
    \caption{Graphs of $t^8 X_{10, 1}(it)$ and $t^9 X_{10, 1}(it)$ on $0.1 \le t \le 5$.}%
    \label{fig:X101}%
\end{figure}

The second proposition shows that monotonicity of $t^{m+1} F'(it)$ implies that of $t^m F(it)$.

\begin{proposition}
    \label{prop:polymod_der}
    Let $F$ be a quasimodular cuspform.
    If $t \mapsto t^{m+1} F'(it)$ is monotone decreasing, then $t \mapsto t^m F(it)$ is also monotone decreasing.
\end{proposition}
\begin{proof}
    For $a > 1$ and $t > 0$, define
    \[
    g_a(t) := F(it) - a^{m} F(iat).
    \]
    Then
    \[
    \frac{\dd}{\dd t} g_a(t) = (-2 \pi) (F'(it) - a^{m+1} F'(iat)),
    \]
    hence $g_a(t)$ is monotone decreasing for all $a > 1$.
    By cuspidality, we have $\lim_{t \to \infty} g_a(t) = 0$ and $g_a(t) \ge 0$, hence $t \mapsto t^m F(it)$ is monotone decreasing.
\end{proof}

\begin{corollary}
    \label{cor:polymod_extd2}
    The following functions are monotone decreasing:
    \begin{align*}
        t \mapsto t^{7} X_{8, 2}(it), \\
        t \mapsto t^{9} X_{10, 2}(it), \\
        t \mapsto t^{11} X_{12, 2}(it), \\
        t \mapsto t^{13} X_{14, 2}(it).
    \end{align*}
\end{corollary}
\begin{proof}
    The identity $X_{8, 2}' = 2 X_{4, 2} X_{6, 1}$ \cite[eq. (46)]{lee2024algebraic} and Lemma \ref{lem:polymod_X42}, Corollary \ref{cor:polymod_X61} implies that $t^{8} X_{8, 2}'(it) = 2 \cdot t^3 X_{4, 2}(it) \cdot t^5 X_{6, 1}(it)$ is monotone decreasing, and Proposition \ref{prop:polymod_der} implies that $t^{7} X_{8, 2}(it)$ is also monotone decreasing.
    The other cases follow from the same argument with the identities $X_{10, 2}' = \frac{8}{9} X_{4, 2} X_{8, 1} + \frac{10}{9} X_{6, 1}^2, X_{12, 2}' = 3 X_{6, 1} X_{8, 2}$, $X_{14, 2}' = 3X_{4, 2} X_{12, 1}$ \cite[eq. (47) (48), (49)]{lee2024algebraic}, Lemma \ref{lem:polymod_X42}, Lemma \ref{lem:X81weak}, Corollary \ref{cor:polymod_X61}, Corollary \ref{cor:polymod_X121}, and Proposition \ref{prop:polymod_der}.
\end{proof}

\subsection{Monotonicity of $t^{w-1}X_{w, 1}(it)$}

Based on Corollary \ref{cor:polymod_X61}, \ref{cor:polymod_X121}, \ref{cor:polymod_X141} and numerical experiments, we conjecture the following monotonicity property of extremal quasimodular forms of depth $1$.

\begin{conjecture}
    \label{conj:Xw1mono}
    For all even $w \ge 6$ but $w = 8, 10$, the function 
    \[
    t \mapsto t^{w-1} X_{w, 1}(it)
    \]
    is monotone decreasing for $t > 0$.
\end{conjecture}
Although we cannot prove this conjecture, we can prove the following weaker claims.
First, the function is monotone decreasing near $t = 0$.

\begin{proposition}
    \label{prop:Xw1monot0}
    For $w \ge 12$, the function $t \mapsto t^{w-1} X_{w, 1}(it)$ is monotone decreasing near $t = 0$.
\end{proposition}
\begin{proof}
    By Corollary \ref{cor:polymod_X121} and \ref{cor:polymod_X141}, we may assume that $w \ge 16$.
    It is enough to show that the derivative is negative near $t = 0$, which is equivalent to
    \[
    2 \pi t X_{w, 1}'(it) - (w-1) X_{w, 1}(it) > 0
    \]
    near $t = 0$.
    By substituting $t$ with $1/t$ and using \eqref{eqn:Xw1trans} and \eqref{eqn:dXw1trans}, this is equivalent to
    \begin{align*}
        &\frac{2\pi}{t} X_{w, 1}'\left(\frac{i}{t}\right) - (w-1) X_{w, 1}\left(\frac{i}{t}\right) \\
        &= (-1)^{\frac{w}{2}} \left[-2 \pi t^{w+1} X_{w,1}'(it) + t^w \left(w X_{w, 1}(it) + 12 B_{w-2}'(it)\right) - \frac{2t^{w-1}}{\pi}\cdot 3(w-1) B_{w-2}(it) \right. \\
        &\quad\left.- (w-1)t^w X_{w, 1}(it) + \frac{6(w-1)t^{w-1}}{\pi}B_{w-2}(it)\right] \\
        &= (-1)^{\frac{w}{2}} t^w \left( -2 \pi t X_{w, 1}'(it) + X_w(it) + 12 B_{w-2}'(it)\right) > 0
    \end{align*}
    for sufficiently large $t$.
    Since $X_{w, 1}(it)$ and $X_{w, 1}'(it)$ are both $O(e^{-4 \pi t})$, the positivity for large enough $t$ is equivalent to $\lim_{t \to \infty} (-1)^{\frac{w}{2}} B_{w-2}'(it) = (-1)^{\frac{w}{2}} \beta_{w-2,1} > 0$, which follows from Lemma \ref{lem:n0coeff} and \ref{lem:n1coeff}.
\end{proof}
Also, we can prove monotonicity of $t^{a_w} X_{w, 1}(it)$ for the smaller exponent $a_w = w - \lceil w/6 \rceil \le w - 1$, using the recurrence relations in Theorem \ref{thm:Xw1rec} (Theorem \ref{thm:weakconj}).

\begin{theorem}
    \label{thm:weakconj}
    For all even $w \ge 6$, the function 
    \[
    t \mapsto t^{w-\lceil w/6 \rceil} X_{w, 1}(it)
    \]
    is monotone decreasing for $t > 0$.
\end{theorem}
\begin{proof}
    Let $a_w := w - \lceil w/6 \rceil$.
    By induction, one can easily check that $a_w$ satisfy the following recurrence relations for all $w \ge 12$ with $w \equiv 0 \pmod{6}$:
    \begin{align}
        a_w &= \min \{a_{w-4} + 4, a_{w-6} + 5\} \label{eqn:awrec1} \\
        a_{w+2} &= \min \{a_{w-2} + 4, a_{w-4} + 5\} \label{eqn:awrec2} \\
        a_{w+4} &= \min \{a_w + 4, a_{w-2} + 5, a_{w-4} + 7\} \label{eqn:awrec3}
    \end{align}
    For example, when $w = 6k$, we have $a_w = 5k$, $a_{w-4} + 4 = 6k - 4 - k + 4 = 5k$, and $a_{w-6} + 5 = 5(k-1) + 5 = 5k$, hence \eqref{eqn:awrec1} holds.
    The other cases can be checked similarly.
    Then the theorem follows for $w = 6, 8, 10$ by Corollary \ref{cor:polymod_X61}, Lemma \ref{lem:X81weak}, and Lemma \ref{lem:X101weak}.
    Assume that it holds for $w -6, w-4, w-2$ for some $w \ge 12$ with $w \equiv 0 \pmod{6}$.
    By induction hypothesis and Lemma \ref{lem:X81weak}, both of the functions
    \[
    t^5 X_{6, 1}(it) \cdot t^{a_{w-4}} X_{w-4, 1}(it), \quad t^6 X_{8, 1}(it) \cdot t^{a_{w-6}} X_{w-6, 1}(it)
    \]
    are monotone decreasing, and by \eqref{eqn:extqmf_rec1} and \eqref{eqn:awrec1}, the function
    \[
    t^{a_w + 1} X_{w, 1}'(it) := \frac{5w}{72} \cdot t^5 X_{6, 1}(it) \cdot t^{a_w - 4} X_{w-4, 1}(it) + \frac{7w}{72} \cdot t^6 X_{8, 1}(it) \cdot t^{a_w - 5} X_{w-6, 1}(it)
    \]
    is also monotone decreasing.
    Therefore, by Proposition \ref{prop:polymod_der}, the function $t^{a_w} X_{w, 1}(it)$ is also monotone decreasing.
    We can prove similarly for $t^{a_{w+2}} X_{w+2, 1}(it)$ and $t^{a_{w+4}} X_{w+4, 1}(it)$ using \eqref{eqn:extqmf_rec2}, \eqref{eqn:extqmf_rec3}, \eqref{eqn:awrec2}, and \eqref{eqn:awrec3}.
\end{proof}

\section{Applications}

\subsection{Positive quasimodular forms of higher levels}
\label{subsec:application_posqmf}

Let $F$ be a quasimodular form of level $1$.
Then for $N \in \bZ_{\ge 1}$, $F^{[N]}(z) := F(Nz)$ is a positive quasimodular form of level $\Gamma_0(N)$ \cite[Proposition 3.9]{lee2024algebraic}.
If we further assume that $F'$ is positive, i.e. $t \mapsto F(it)$ is monotone decreasing for $t > 0$, then $F(z) - F^{[N]}(z) = F(z) - F(Nz)$ is also positive.
The following proposition shows that one can obtain a stronger result if the monotonicity of $t \mapsto t^m F(it)$ is known for some $m > 0$.

\begin{corollary}
    \label{cor:polymod_posdif}
    Let $F$ be a quasimodular form and let $m \in \bR_{> 0}$.
    Assume that $t \mapsto t^m F(it)$ is monotone decreasing for $t > 0$.
    Then for $N \in \bZ_{\ge 1}$, $F(z) - N^{m} F(Nz)$ is a positive quasimodular form of level $\Gamma_0(N)$.
    Also, if $\lim_{t \to 0^+} t^m F(it)$ exists and is positive, then the constant $N^{m}$ is optimal, i.e. for any $c > N^{m}$, $F(z) - c F(Nz)$ is not positive.
\end{corollary}
\begin{proof}
    Since $t^{m} F(it)$ is monotone decreasing,
    \[t^m F(it) - (Nt)^m F(iNt) = t^m (F(it) - N^m F(iNt))\]
    is positive for $t > 0$.
    If $\lim_{t \to 0^+} t^m F(it) = L > 0$, then
    \[
    \lim_{t \to 0^+} \frac{F(it)}{F(iNt)} = \lim_{t \to 0^+} \frac{N^m \cdot t^m F(it)}{(Nt)^m F(iNt)} = \frac{N^m \cdot L}{L} = N^m
    \]
    which shows the optimality of the constant.
\end{proof}

\begin{example}
    \label{ex:polymod_level}
    From Lemma \ref{lem:polymod_X42}, \ref{lem:polymod_D2} and Corollary \ref{cor:polymod_X61}, \ref{cor:polymod_X121}, we obtain the following positive quasimodular forms by taking $N = 2$ in Corollary \ref{cor:polymod_posdif}:
    \begin{align*}
        P_1(z) := X_{4, 2}(z) - 8 X_{4, 2}(2z) &= \sum_{n \ge 1} n \left(\sigma_1(n) - 4 \sigma_1\left(\frac{n}{2}\right)\right) q^n \in \QM_{4}^{2, +}(\Gamma_0(2)) \\
        P_2(z) := \frac{-E_2(z) + 5E_2(2z) - 4E_2(4z)}{24} &= \sum_{n \ge 1} \left(\sigma_1(n) - 5 \sigma_1\left(\frac{n}{2}\right) + 4\sigma_1\left(\frac{n}{4}\right)\right) q^n \in \QM_{2}^{1, +}(\Gamma_0(4)) \\
        P_3(z) := X_{6, 1}(z) - 32 X_{6, 1}(2z) &= \sum_{n \ge 1} n \left(\sigma_3(n) - 16 \sigma_3\left(\frac{n}{2}\right)\right) q^n \in \QM_{6}^{1, +}(\Gamma_0(2)) \\
        P_4(z) := X_{12, 1}(z) - 2^{11} X_{12, 1}(2z) &= \frac{1}{1050}\sum_{n \ge 1} \left[n\left(\sigma_9(n) - 2^{10} \sigma_9\left(\frac{n}{2}\right)\right) - \left(\tau(n) - 2^{11}\tau\left(\frac{n}{2}\right)\right) \right]q^n \in \QM_{12}^{1, +}(\Gamma_0(2)) \\
    \end{align*}
    Note that none of these quasimodular forms are completely positive.
    For example, the $n$-th coefficient of $P_1(z)$ for odd $n$ is $n \sigma_1(n)$, which is positive.
    For even $n = 2^k \cdot m$ with $k \ge 1$ and $m$ odd is 
    \[
    n\left(\sigma_1(n) - 4\sigma_1\left(\frac{n}{2}\right)\right) = n \cdot (\sigma_1(2^k) \sigma_1(m) - 4 \sigma_1(2^{k-1}) \sigma_1(m)) = n \cdot (-2^{k+1} + 3) \sigma_1(m) < 0.
    \]
    In particular, taking the negation of the form with substitution $z \mapsto z + \frac{1}{2}$ gives a completely positive quasimodular form of level $\Gamma_0(4)$:
    \[
    - P_1\left(z + \frac{1}{2}\right) = -X_{4, 2} \left(z + \frac{1}{2}\right) + 8 X_{4, 2} \left(2z\right) = \sum_{n \ge 1} n \left(\sigma_1(n) - 4 \sigma_1\left(\frac{n}{2}\right)\right) (-1)^{n+1} q^n \in \QM_{4}^{2, ++}(\Gamma_0(4)).
    \]
    Similarly, the $n$-th coefficient of $P_3(z)$ is positive if and only if $n$ is odd, and we get a completely positive quasimodular form of level $\Gamma_0(4)$:
    \[
    - P_3\left(z + \frac{1}{2}\right) = -X_{6, 1} \left(z + \frac{1}{2}\right) + 32 X_{6, 1} \left(2z\right) = \sum_{n \ge 1} n \left(\sigma_3(n) - 16 \sigma_3\left(\frac{n}{2}\right)\right) (-1)^{n+1} q^n \in \QM_{6}^{1, ++}(\Gamma_0(4)).
    \]
    For $P_2(z)$, one can check that the $n$-th coefficient is positive for odd $n$, and negative for $n = 2m$ for odd $m$.
    When $n = 2^k \cdot m$ with $k \ge 2$, the $n$-th coefficient is
    \[
    (\sigma_1(2^k) - 5 \sigma_1(2^{k-1}) + 4 \sigma_1(2^{k-2})) \sigma_1(m) = ((2^{k+1} - 1) - 5 (2^k - 1) + 4 (2^{k-1} - 1))\sigma_1(m) = 3 \cdot 2^{k-1} \sigma_1(m) > 0.
    \]
    Thus $n$-th coefficient of $P_2(z)$ is positive if and only if $n \not \equiv 2 \pmod{4}$.
    If we write $P_2(z) = \sum_{n \ge 1} a_n q^n$, then
    \begin{align*}
        P_2\left(\frac{z}{2} + \frac{1}{4}\right) &= i a_1 q^{1/2} - a_2 q - i a_3 q^{3/2} + a_4 q^2 + i a_5 q^{5/2} - a_6 q^3 - i a_7 q^{7/2} + a_8 q^4 + \cdots \\
        P_2\left(\frac{z}{2} - \frac{1}{4}\right) &= -i a_1 q^{1/2} - a_2 q + i a_3 q^{3/2} + a_4 q^2 - i a_5 q^{5/2} - a_6 q^3 + i a_7 q^{7/2} + a_8 q^4 + \cdots
    \end{align*}
    This shows that the following quasimodular forms are completely positive:
    \begin{align*}
        &\frac{1}{2} \left(P_2(z) - P_2\left(z + \frac{1}{2}\right)\right) = \frac{- E_2(z) + E_2\left(z + \frac{1}{2}\right)}{48} = \sum_{k \ge 0} \sigma_1(2k + 1) q^{2k + 1}  \\
        &\frac{1}{14} \left(P_2\left(\frac{z}{2} + \frac{1}{4}\right) + P_2 \left(\frac{z}{2} - \frac{1}{4}\right)\right) = \frac{1}{24} (6 E_2(4z) - 5 E_2(2z) - E_2(z))  \\
        &= \frac{1}{2} \sum_{n \equiv 0\,(\mathrm{mod}{4})} \left(\sigma_1(n) - 5 \sigma_1\left(\frac{n}{2}\right) + 4 \sigma_1\left(\frac{n}{4}\right)\right) q^{n/2} + \frac{1}{2} \sum_{n \equiv 2\,(\mathrm{mod}{4})} \left(5 \sigma_1\left(\frac{n}{2}\right) - \sigma_1(n)\right)q^{n/2} 
    \end{align*}
    where the simplification is done using \eqref{eqn:e2_level2}.
    Both forms are of level $\Gamma_0(4)$.

    In case of $P_4(z)$, we can use Deligne's bound for $\tau(n)$ to study signs of the coefficients.
    The $n$-th coefficient of $P_4(z)$ for odd $n$ is positive, since
    \[
    n \sigma_9(n) - \tau(n) > n \sigma_9(n) - \sigma_0(n) n^{\frac{11}{2}} \ge n^{10} - n^{\frac{13}{2}} > 0.
    \]
    by the trivial estimates $\sigma_9(n) > n^9$, $\sigma_0(n) \le n$ and Deligne's bound $|\tau(n)| < \sigma_0(n) n^{\frac{11}{2}}$.
    If $n = 2^k \cdot m$ with $k \ge 1$ and $m$ odd, then we can estimate the $n$-th coefficient using Deligne's bound again:
    \begin{align*}
        &n \left(\sigma_9(2^k) - 2^{10} \sigma_9(2^{k-1})\right) \sigma_9(m) - \left(\tau(2^k) - 2^{11} \tau(2^{k-1})\right) \tau(m) \\
        &= n \cdot \frac{-2^{9(k+1)} + 1023}{511} \cdot \sigma_9(m) - (\tau(2^{k+1}) + 25 \tau(2^{k})) \tau(m) \\
        &< 2^k  \cdot \frac{-2^{9(k+1)} + 1023}{511} \cdot m \sigma_9(m) + ((k+2) 2^{\frac{11(k+1)}{2}} + 25 (k+1) 2^{\frac{11k}{2}}) m^{\frac{13}{2}}
    \end{align*}
    and the last term is negative for $k \ge 2$, since
    \begin{align*}
        2^k \cdot \frac{2^{9(k+1)} - 1023}{511} >  2^{\frac{11k}{2}} ((2^{\frac{11}{2}} + 25)k + (2^{\frac{13}{2}} + 25)), \quad m \sigma_9(m) > m^{10} \ge m^{\frac{13}{2}}
    \end{align*}
    holds for $k \ge 2$ and all odd $m \ge 1$.
    Note that the recurrence relation $\tau(2^{k+1}) = \tau(2)\tau(2^k) - 2^{11} \tau(2^{k-1}) = -24 \tau(2^k) - 2048 \tau(2^{k-1})$ is also used in the above estimation.
    When $k=1$ and $n = 2m$, the $n$-th coefficient is
    \[
    n (\sigma_9(2) - 2^{10} \sigma_9(1)) - (\tau(2) - 2^{11}\tau(1)) \tau(m) = -1022 m \sigma_9(m) + 2072 \tau(m)
    \]
    which is positive for $m = 1$, and negative for $m \ge 2$ by Deligne's bound.
\end{example}

\begin{example}
    By Theorem \ref{thm:weakconj}, for each even $w \ge 6$ and $N \ge 2$,
    \[
    X_{w, 1}(z) - N^{w - \lceil w/6 \rceil} X_{w, 1}(Nz)
    \]
    is a positive quasimodular form of level $\Gamma_0(N)$.
    If Conjecture \ref{conj:Xw1mono} holds, then we can take larger exponent $w - 1$ instead of $w - \lceil w/6 \rceil$.
\end{example}

\begin{example}
    \label{ex:polymod_level2_d2}
    From Corollary \ref{cor:polymod_extd2}, we obtain the following positive quasimodular forms:
    \begin{align*}
        \widetilde{X}_{8, 2}(z) := X_{8, 2}(z) - 2^{7} X_{8, 2}(2z) & \in \QM_{8}^{2, +}(\Gamma_0(2)) \\
        \widetilde{X}_{10, 2}(z) := X_{10, 2}(z) - 2^{9} X_{10, 2}(2z) & \in \QM_{10}^{2, +}(\Gamma_0(2)) \\
        \widetilde{X}_{12, 2}(z) := X_{12, 2}(z) - 2^{11} X_{12, 2}(2z) & \in \QM_{12}^{2, +}(\Gamma_0(2)) \\
        \widetilde{X}_{14, 2}(z) := X_{14, 2}(z) - 2^{13} X_{14, 2}(2z) & \in \QM_{14}^{2, +}(\Gamma_0(2)).
    \end{align*}
\end{example}

One may take smaller exponent $m$ so that $F(z) - N^m F(Nz)$ becomes completely positive. Such example can be found when Fourier expansions of $F$ can be easily expressed in terms of divisor sums or Ramanujan tau function - see Appendix \ref{appendix:comppos} for some examples of this type.

\subsection{Another proof of a modular form inequality for optimality of Leech lattice packing}
\label{subsec:application_packingineq}

By using the positive quasimodular forms constructed in the previous subsection, we can give an alternative proof of one of the modular form inequalities \cite[Lemma A.2]{cohn2017sphere} used in the proof of optimality of Leech lattice packing.

Define $F \in \QM_{16}^{2}(\SL_2(\bZ))$ and $G \in \cM_{14}(\Gamma(2))$ as
\begin{align}
    F &:= 49 E_2^2 E_4^3 - 25 E_2^2 E_6^2 - 48 E_2 E_4^2 E_6 - 25 E_4^4 + 49 E_4 E_6^2 \label{eqn:F} \\
    G &:= H_2^5 (2 H_2^2 + 7 H_2 H_4 + 7 H_4^2) \label{eqn:G}
\end{align}
where $H_2 = \Theta_2^4$ and $H_4 = \Theta_4^4$ are the fourth powers of Jacobi theta functions.
Then the inequality \cite[Lemma A.2]{cohn2017sphere} is equivalent to
\[
F(it) > \frac{432}{\pi^2} G(it)
\]
for $t > 0$ \cite[eq. (60)]{lee2024algebraic}.
The author gave an algebraic proof of this inequality, by proving monotonicity of the quotient $F(it)/G(it)$ \cite[Proposition 6.5]{lee2024algebraic} and showing that the limit $\lim_{t \to 0^+} F(it)/G(it) = 432/\pi^2$ \cite[Proposition 6.4]{lee2024algebraic}.
The monotonicity part was proved by using the differential equations satisfied by $F$ and $G$ \cite[eq. (64), (65)]{lee2024algebraic}.
In this section, we give another proof of the monotonicity, by using the positive quasimodular forms constructed in the previous subsection.

It is enough to show that
\[
\cL_{1, 0} := F' G - F G' = (\partial_{14} F) G - F (\partial_{14} G)
\]
is positive.
Using Sage, one can check that $\cL_{1, 0}$ can be factored as
\begin{equation}
\label{eqn:d24ineqfactor}
    \cL_{1, 0} = \frac{105}{8} \cdot H_{2}^{5} H_{4}^{2} (H_{2} + H_{4})^{2} \cdot L,
\end{equation}
where $L := K_{10}E_2^2 + K_{12}E_2 + K_{14}$ is a quasimodular form of weight $14$, level $\Gamma(2)$, and depth $2$ with
\begin{align*}
    K_{10} &= -2 (23 H_2^4 + 46 H_2^3 H_4 + 54 H_2^2 H_4^2 + 16 H_2 H_4^3 + 8 H_4^4) (H_2 + 2H_4), \\
    K_{12} &= -2 (10 H_2^4 + 35 H_2^3 H_4 + 3 H_2^2 H_4^2 -64 H_2 H_4^3 -32 H_4^4)(H_2^2 + H_2 H_4 + H_4^2), \\
    K_{14} &= (26 H_2^6 + 78 H_2^5 H_4 + 177 H_2^4 H_4^2 + 182 H_2^3 H_4^3 + 51 H_2^2 H_4^4 - 48 H_2 H_4^5 - 16 H_4^6)(H_2 + 2H_4).
\end{align*}
Here $K_{w}$'s for $w \in \{10, 12, 14\}$ are weight $w$, level $\Gamma(2)$ modular forms.
We observe that $K_{w}$'s are all invariant under $z \mapsto z + 1$: in other words, they are modular forms of level $\Gamma_0(2) \subset \Gamma(2)$.
The ring of modular forms of level $\Gamma_0(2)$ is generated by two elements $A = H_2^2$ and $B = H_2 + 2H_4$, where both have positive Fourier coefficients:
\[
    A = 4 \left(\sum_{n \geq 0} r_4(2n+1) q^{n + \frac{1}{2}}\right)^{2}, \quad B = 2\left(1 + \sum_{n \geq 1}r_4(2n) q^{n}\right).
\]
(Here $r_4(n)$ is the number of ways to write $n$ as a sum of four squares.)
Let $X_{w, s}^{[2]}(z) := X_{w, s}(2z)$.
One can write $L$ as a combination of $X_{8, 2}^{[2]}, \widetilde{X}_{8, 2}, X_{10, 2}^{[2]}, \widetilde{X}_{10, 2}, X_{12, 2}^{[2]} , \widetilde{X}_{12, 2}$ and $A, B$ by:
\begin{equation}
\label{eqn:d24ineqLcomb}
    L = a_1 X_{8, 2}^{[2]}AB + a_2 \widetilde{X}_{8, 2}AB + a_3 X_{10, 2}^{[2]}A + a_4 \widetilde{X}_{10, 2}A + a_5 X_{12, 2}^{[2]} B + a_6 \widetilde{X}_{12, 2} B
\end{equation}
where all the coefficients $a_1, \dots, a_6$ are positive:
\begin{align*}
    a_1 = 78278400,  \quad
    a_2 = 550800, \quad
    a_3 = 90823680, \quad
    a_4 = 116640, \quad
    a_5 = 678813696000, \quad
    a_6 = 331776000,
\end{align*}
and the positivity of $L$ follows from those of $X_{8, 2}, X_{10, 2}, X_{12, 2}$ \cite[Proposition 4.6]{lee2024algebraic} and Example \ref{ex:polymod_level2_d2}.

\subsection{Algebraic proof of a modular form inequality for universal optimality of Leech lattice}
\label{subsec:application_universalineq}

The last application is the new algebraic proof of the inequality \eqref{eqn:inequality3b}, which was the main motivation of this work.
The inequality is closely related to the extremal quasimodular form $X_{12, 1}$, and becomes a direct consequence of Corollary \ref{cor:polymod_X121}.
\begin{lemma}
    \label{lem:inequality3b}
    The inequality \eqref{eqn:inequality3b} is equivalent to the monotonicity of the function $t \mapsto t^{11} X_{12, 1}(it)$ on $(0, \infty)$.
    Especially, it follows from Corollary \ref{cor:polymod_X121}.
\end{lemma}
\begin{proof}
    Recall that (see \cite[Table 1]{lee2024algebraic})
    \[
    X_{12, 1} = \frac{-12 E_2 E_4 E_6 + 5 E_4^3 + 7 E_6^2}{3991680} \Leftrightarrow-12 E_2 E_4 E_6 + 5 E_4^3 + 7 E_6^2 = 3991680 X_{12, 1}.
    \]
    By \cite[Theorem 4.3]{lee2024algebraic}, we have
    \[
    X_{12, 1}' = 2 X_{6, 1} X_{8, 1} = 2 \cdot \frac{E_2 E_4 - E_6}{720} \cdot \frac{- E_2 E_6 + E_4^2}{1008} = \frac{11}{3991680} (-E_2^2 E_4 E_6 + E_2 E_4^3 + E_2 E_6^2 - E_4^2 E_6).
    \]
    Using this, one can express $\what{\cE}_{1, 0}$ and $\what{\cE}_{1, 1}$ in terms of $X_{12, 1}$ and $X_{12, 1}'$:
    \begin{align*}
        \what{\cE}_{1, 0} &= 12 \pi i \left(- E_2 E_4 E_6 + E_6^2 + \frac{5}{12} (E_4^3 - E_6^2)\right) = \pi i \cdot (12 E_2 E_4 E_6 + 5 E_4^3 + 7 E_6^2) = \pi i \cdot 3991680 X_{12, 1} \\
        \what{\cE}_{1, 1} &= 2 \pi^2 (E_2^2 E_4 E_6 - 2 E_2 E_6^2 - E_2 (E_4^3 - E_6^2) + E_4^2 E_6) = -\frac{2 \pi^2}{11} \cdot 3991680 X_{12, 1}'
    \end{align*}
    and
    \[
        i\what{\cE}_{1}(it) = -t \what{\cE}_{1, 1}(it) + i \what{\cE}_{1, 0}(it) = 3991680 \pi \left(\frac{2 \pi t}{11} X_{12, 1}'(it) - X_{12, 1}(it)\right) = -362880 \pi t^{-10} \cdot \frac{\dd}{\dd t}(t^{11} X_{12, 1}(it))
    \]
    which proves the claim.
\end{proof}

Figure \ref{fig:3b} shows the graphs of $t \mapsto t^{11} X_{12, 1}(it)$ and $t \mapsto X_{12, 1}(it) / X_{12, 1}'(it)$, where the former is monotone decreasing and the latter is monotone increasing.
By \eqref{eqn:extremallimit2}, we have
\[
\lim_{t \to 0^+} t^{11} X_{12, 1}(it) = \frac{6(-1)^{\frac{12}{2}}\beta_{10}}{\pi} = \frac{1}{55400 \pi}
\]
and the tangent line of the latter at $t = 0$ is $l(t) := \frac{2\pi t}{11}$.
We leave it as future work to give simpler proofs of the other inequalities (1), (2), and (3a) of \cite[p. 1067]{cohn2022universal}.

\begin{figure}[h]
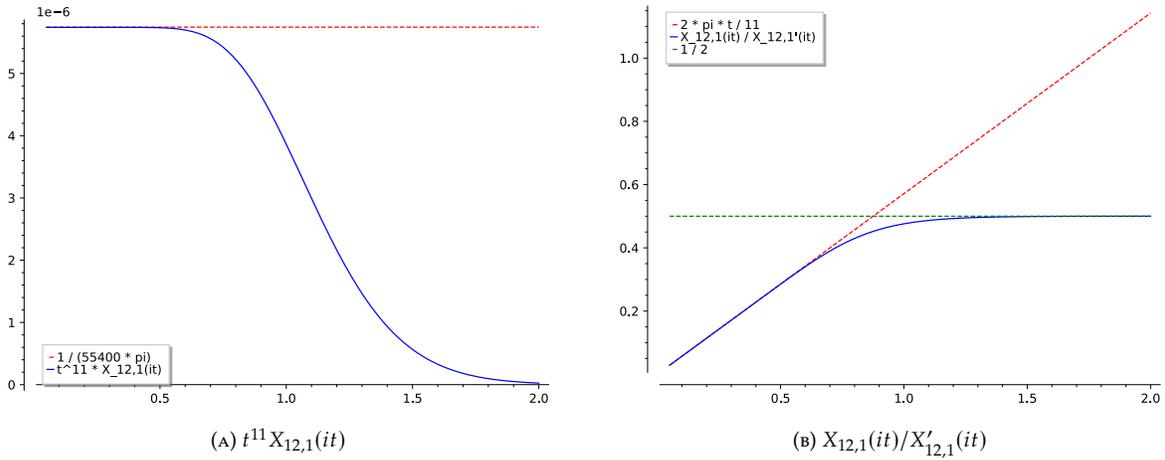
%
    \centering
    \subfloat[\centering $t^{11}X_{12, 1}(it)$]{{\includesvg[width=0.45\textwidth]{images/3b_1.svg}}}%
    \qquad
    \subfloat[\centering $X_{12, 1}(it) / X_{12, 1}'(it)$]{{\includesvg[width=0.45\textwidth]{images/3b_2.svg}}}%
    \caption{Graphs of $t \mapsto t^{11} X_{12, 1}(it)$ and $t \mapsto X_{12, 1}(it) / X_{12, 1}'(it)$ with asymptotes as $t \to 0^+$ and $t \to \infty$.}%
    \label{fig:3b}%
\end{figure}

\section{Conclusion}

In this paper, we studied monotonicity of the functions of the form $t \mapsto t^m F(it)$ where $F$ is a quasimodular form and $m > 0$, and presented several applications of it.
In particular, we show that one can construct positive quasimodular forms of level $> 1$ using the monotonicity, where most of the examples we found in Section \ref{subsec:application_posqmf} are not completely positive.
We computed the natural density of the set $\{n : a_n > 0\}$ for the Fourier coefficients of these examples, and it is natural to ask if the density is bounded from below by a positive constant for every positive quasimodular form.

\begin{question}
    \label{question:density}
    Is there a positive constant $c \in (0, 1)$ such that for every positive quasimodular form $F = \sum_{n \ge n_0} a_n q^n$, the natural density of the set $\{n : a_n > 0\}$ is at least $c$?
\end{question}

In Example \ref{ex:polymod_level}, the natural density of the set $\{n : a_n > 0\}$ for each $P_1, P_2, P_3, P_4$ is $\frac{1}{2}, \frac{3}{4}, \frac{1}{2}, \frac{1}{2}$ respectively, which shows $c \le \frac{1}{2}$.
We conjecture that $c = \frac{1}{2}$ is indeed the answer to Question \ref{question:density}.
\begin{conjecture}
    \label{conj:density}
    For every positive quasimodular form $F = \sum_{n \ge n_0} a_n q^n$, the natural density of the set $\{n : a_n > 0\}$ exists and is at least $\frac{1}{2}$.
\end{conjecture}
Kane, Krishnamoorthy, and Lau \cite{kane2025conjecture} showed that, when $F$ is a linear combination of cusp forms or their derivatives, prime-indexed coefficients have infinitely many sign changes.
In particular, there is no completely positive quasimodular form of the form $\sum_{j=0}^{s} D^j F_j$ with all $F_j$ cusp forms.
However, it is still possible that such quasimodular form is positive.
For example, consider the weight $16$ and depth $2$ quasimodular form
\[
F = X_{4, 2} \Delta = \frac{1}{312} (E_4 \Delta - \Delta'') = q^2 - 18 q^3 + 1240 q^4 - 220 q^5 - 1620 q^6 + 11676 q^7 + \cdots.
\]
This is a linear combination of (derivatives of) cusp forms which is also positive.
It seems that the natural density of the set $\{n : a_n > 0\}$ for $F$ is $\frac{1}{2}$.

\appendix

\section{Completely positive quasimodular forms of higher levels}
\label{appendix:comppos}

In Section \ref{subsec:application_posqmf}, we constructed positive quasimodular forms of level $> 1$ by using the monotonicity of the functions of the form \eqref{eqn:tmf}.
Here we give some examples of \emph{completely} positive quasimodular forms of level $>1$, directly constructed from level 1 quasimodular forms by analyzing the Fourier coefficients.

Let $F(z) = \sum_{n \ge 1} a_n q^n$ be a quasimodular form of level $1$.
Then $F(z) - c F(Nz) = \sum_{n \ge 1} (a_n - c a_{n/N}) q^n$ is completely positive if and only if
\begin{equation}
    \inf_{n \ge 1} \frac{a_{Nn}}{a_n} \ge c.
    \label{eqn:infcoeffrat}
\end{equation}
In particular, the infimum \eqref{eqn:infcoeffrat} is the largest possible constant $c$ such that $F(z) - c F(Nz)$ remains completely positive.

\begin{proposition}
    \label{prop:extYpos}
    For given weight $w$ and depth $s$, define
    \[
        Y_{w, s}(z) := X_{w, s}(z) - 2^{w-s} X_{w, s}(2z).
    \]
    Then $Y_{4, 2}, Y_{8, 2}, Y_{10, 2}$ and $Y_{12,2}$ are completely positive.
\end{proposition}

\begin{proof}
Let $a_{w, n}$ be the $n$-th Fourier coefficient of $X_{w, 2}$.
It is enough to show that
\begin{equation}
\label{eqn:infcoeffrat2}
    \inf_{n} \frac{a_{w, 2n}}{a_{w, n}} \geq 2^{w-2}.
\end{equation}
We can express the extremal forms in terms of (derivatives of) Eisenstein series and obtain explicit formulas for the Fourier coefficients as
\begin{align}
    X_{4, 2} &= -\frac{E_2'}{24} = \sum_{n \geq 1} n\sigma_1(n) q^n, \\
    X_{8, 2} &= -\frac{E_6'}{15120} - \frac{E_4''}{7200} = \sum_{n \geq 2} \left(\frac{n \sigma_5(n) - n^2 \sigma_3(n)}{30}\right) q^n, \\
    X_{10, 2} &= \frac{(E_4^2)'}{60480} + \frac{E_6''}{63504}  = \sum_{n \geq 2} \left(\frac{n \sigma_7(n) - n^2 \sigma_5(n)}{126}\right) q^n.
\end{align}
(Note that $E_4^2 = 1 + 480 \sum_{n \geq 1} \sigma_7(n)q^n$ is the weight $8$ normalized Eisenstein series.)
Using these, we can compute the infimum \eqref{eqn:infcoeffrat2} precisely.
Recall that the divisor sum functions are multiplicative.
Write $n = 2^k m$ for $k\geq 0$ and odd $m$.
Then the infimum for $w = 4$ can be computed as
\begin{align*}
    \inf_{n\geq 1} \frac{2n\sigma_1(2n)}{n \sigma_1(n)} = \inf_{n \geq 1} \frac{2^{k+1}m \sigma_1(2^{k+1})\sigma_1(m)}{2^k m \sigma_1(2^k)\sigma_1(m)} = \inf_{k \geq 0} \frac{2(2^{k+2} - 1)}{2^{k+1} - 1} = 4 = 2^{4-2}.
\end{align*}
The cases $w = 8$ and $w = 10$ are similar but slightly more complicated.
When $w = 8$, we first observe that $n \sigma_5(n) - n^2 \sigma_3(n) \geq 0 \Leftrightarrow b_{n} := \frac{\sigma_5(n)}{n \sigma_3(n)} \geq 1$ for all $n \geq 2$, due to the complete positivity of $X_{8, 2}$\footnote{The inequality also follows from the trivial estimate $n^a \leq \sigma_a(n) \leq 1^a + 2^a + \cdots + n^a$.}. 
Again, we write $n = 2^k m$ with odd $m$. Then
\begin{align*}
    \frac{a_{8, 2n}}{a_{8, n}} = \frac{2n \sigma_5(2n) - 4n^2 \sigma_3(2n)}{n \sigma_5(n) - n^2 \sigma_3(n)} 
    = \frac{2^{k+1}m \sigma_5(2^{k+1})\sigma_5(m) - 2^{2k+2}m^2 \sigma_3(2^{k+1})\sigma_3(m)}{2^k m \sigma_5(2^k)\sigma_5(m) - 2^{2k}m^2 \sigma_3(2^k)\sigma_3(m)} 
    = \frac{4\sigma_3(2^{k+1})}{\sigma_3(2^k)} \cdot \frac{b_{2^{k+1}} b_{m} - 1}{b_{2^k} b_{m} - 1}.
\end{align*}
It is easy to check that $b_{2^k}$ is increasing in $k \geq 0$.
Hence
\begin{align*}
    \frac{a_{8, 2n}}{a_{8, n}} \geq \frac{4 \sigma_3(2^{k+1})}{\sigma_3(2^k)} \cdot \frac{b_{2^{k+1}}}{b_{2^k}} = \frac{2 \sigma_5(2^{k+1})}{\sigma_5(2^k)} = \frac{2(2^{5(k+2)} - 1)}{2^{5(k+1)} - 1} > 2^{6} = 2^{8-2}.
\end{align*}
The constant $2^6$ is also optimal by considering $n = 2^k$ with $k \to \infty$.
The same argument also works for $w = 10$.

For $w = 12$, $X_{12, 2}$ can be expressed as
\[
X_{12, 2} = \frac{1}{18000} \left(\frac{17 \Delta}{21} - \frac{E_{10}'}{308} - \frac{E_8''}{288}\right) = \frac{1}{18000} \sum_{n \ge 3}\left(\frac{17 \tau(n)}{21} + \frac{6n\sigma_9(n)}{7} - \frac{5n^2 \sigma_7(n)}{3}\right) q^n =: \sum_{n \ge 3} c_n q^n,
\]
so it is enough to show that $c_{2n} \ge 2^{10} c_n$ for $n \ge 2$, or equivalently,
\begin{equation}
    \frac{17 \tau(2n)}{21} + \frac{12n\sigma_9(2n)}{7} - \frac{10 n^2 \sigma_7(2n)}{3} - 2^{10} \left(\frac{17 \tau(n)}{21} + \frac{6 n \sigma_9(n)}{7} - \frac{5 n^2 \sigma_7(n)}{3}\right) \ge 0. \label{eqn:X122coeffineq}
\end{equation}
Write $n = 2^k m$ with odd $m$.
Then \eqref{eqn:X122coeffineq} is equivalent to
\begin{align}
    &\frac{17 \tau(2^{k+1}) \tau(m)}{21} + \frac{12 \cdot 2^{k} m \sigma_9(2^{k+1}) \sigma_9(m)}{7} - \frac{5 \cdot 2^{2k+2} m^2 \sigma_7(2^{k+1}) \sigma_7(m)}{3} \nonumber \\
    &\quad - 2^{10} \left(\frac{17 \tau(2^k) \tau(m)}{21} + \frac{6 \cdot 2^k m \sigma_9(2^k) \sigma_9(m)}{7} - \frac{5 \cdot 2^{2k} m^2 \sigma_7(2^k) \sigma_7(m)}{3}\right) \nonumber \\
    &= \frac{17\tau(m)}{21} (\tau(2^{k+1}) - 1048 \tau(2^k)) \nonumber \\
    &\quad + \frac{6 \cdot 2^k m \sigma_9(m)}{7} (2 \sigma_9(2^{k+1}) - 2^{10} \sigma_9(2^k)) - \frac{5 \cdot 2^{2k} m^2 \sigma_7(m)}{3} (2^2 \sigma_7(2^{k+1}) - 2^{10} \sigma_7(2^k)) \nonumber \\
    &= \frac{17\tau(m)}{21} (\tau(2^{k+1}) - 1048 \tau(2^k)) \nonumber \\
    &\quad + \frac{6 \cdot 2^{k+1} m \sigma_9(m)}{7} \left(\frac{2^{9(k+2)}-1}{511} - 512 \cdot \frac{2^{9(k+1)}-1}{511}\right) - \frac{5 \cdot 2^{2k+2} m^2 \sigma_7(m)}{3} \left(\frac{2^{7(k+2)}-1}{127} - 256 \cdot \frac{2^{7(k+1)}-1}{127}\right) \nonumber \\
    &=\frac{5}{3} \cdot \frac{128}{127} m^2 \sigma_7(m) 2^{2k+2} \left(2^{7(k+1)} - \frac{255}{128}\right) + \frac{17}{21} \tau(m) \left(\tau(2^{k+1}) - 1048 \tau(2^{k})\right) + \frac{6}{7} m \sigma_9(m) 2^{k+1} \label{eqn:x122coeff}
\end{align}

When $k = 0$, $m \ge 3$ (since $n \ge 2$) and \eqref{eqn:x122coeff} can be bounded from below by
\begin{align}
\frac{2520}{3} m^2 \sigma_7(m) - \frac{17 \cdot 1072}{21} \tau(m) + \frac{12}{7} m\sigma_9(m) > \frac{2520}{3} m^9 - \frac{18224}{21} \sigma_0(m)m^{\frac{11}{2}} + \frac{12}{7} m^{10} \label{eqn:appendix_ineq1}
\end{align}
and the right hand side can be easily checked to be positive for $m \ge 3$ using the estimate $\sigma_0(m) \le m$.
When $k \ge 1$, since the last term of \eqref{eqn:x122coeff} is positive, it is enough to show that
\[
m^2 \sigma_7(m) > |\tau(m)|\quad \text{and}\quad \frac{5}{3} (512 \cdot 2^{9k} - 8 \cdot 2^{2k}) > \frac{17}{21} |\tau(2^{k+1}) - 1048 \tau(2^k)| 
\]
for odd $m$ and $k \ge 0$.
The first inequality easily follows from Deligne's bound and the trivial estimates $\sigma_7(m) \ge m^7$ and $m \ge \sigma_0(m)$:
\[
m^2 \sigma_7(m) \ge m^2 \cdot m^7 = m^9 \ge m^{\frac{13}{2}} = m \cdot m^{\frac{11}{2}} \ge \sigma_0(m) m^{\frac{11}{2}} > |\tau(m)|.
\]
The second inequality can be also checked by using the Deligne's bound $|\tau(2^k)| \le \sigma_0(2^k) 2^{\frac{11}{2}k} = (k+1) 2^{\frac{11}{2}k}$:
\begin{align}
    &\frac{5}{3} (512 \cdot 2^{9k} - 8 \cdot 2^{2k}) - \frac{17}{21} (|\tau(2^{k+1})| + 1048 |\tau(2^{k})|) \nonumber \\
    &\ge \frac{2560}{3} \cdot 2^{9k} - \frac{40}{3} \cdot 2^{2k} - \frac{17}{21} ((k+2) 2^{\frac{11}{2}(k+1)} + 1048 (k+1) 2^{\frac{11}{2}k}) \nonumber \\
    &\ge \frac{2560}{3} \cdot 2^{9k} - \frac{40}{3} \cdot 2^{2k} - \frac{17}{21} ((2^{\frac{11}{2}} + 1048)k + (2^{\frac{13}{2}} + 1048)) 2^{\frac{11}{2}k} \label{eqn:appendix_ineq2}
\end{align}
which is positive for $k \ge 1$.
\end{proof}
Using these, we can reprove that the quasimodular form $L$ in Section \ref{subsec:application_packingineq} is (completely) positive by writing it as
\[
L = a_1' X_{8, 2}^{[2]} AB + a_2' Y_{8, 2} AB + a_3' X_{10, 2}^{[2]} A + a_4' Y_{10, 2} A + a_5' X_{12, 2}^{[2]} + a_6' Y_{12, 2}
\]
where
\begin{align*}
    a_1' = 43027200, \quad
    a_2' = 550800, \quad
    a_3' = 60963840, \quad
    a_4' = 116640, \quad
    a_5' = 339075072000, \quad
    a_6' = 331776000.
\end{align*}

In fact, we conjecture that $Y_{w, 2}$ is completely positive for \emph{all} $w \ne 6$.

\begin{conjecture}
    \label{conj:Yw2cpos}
    For all $w \ne 6$, $Y_{w, 2}(z) = X_{w, 2}(z) - 2^{w-2} X_{w, 2}(2z)$ is completely positive.
\end{conjecture}

This conjecture is stronger than the Kaneko--Koike's conjecture on the complete positivity of $X_{w, 2}$, since we can write $X_{w, 2}$ as
\[
    X_{w, 2}(z) = Y_{w, 2}(z) + 2^{w-2} Y_{w, 2}(2z) + 2^{2(w-2)} Y_{w, 2}(4z) + 2^{3(w-2)} Y_{w, 2}(8z) + \cdots.
\]
Here are $q$-expansions of $Y_{w, 2}$ for $w \le 16$:
\begin{align*}
    Y_{4, 2} &= q + 2q^2 + 12 q^3 + 4 q^4 + 30 q^5 + \cdots, \\
    Y_{8, 2} &= q^2 + 16 q^3 + 38 q^4 + 416 q^5 + 284 q^6 + \cdots, \\
    Y_{10, 2} &= q^2 + \frac{104}{3} q^3 + 134 q^4 + 2480 q^5 + \frac{6796}{3} q^6 + \cdots, \\
    Y_{12, 2} &= q^3 + \frac{51}{2} q^4 + \frac{1422}{5} q^5 + 920 q^6 + 9714 q^7 + \cdots, \\
    Y_{14, 2} &= q^3 + \frac{93}{2} q^4 + 810 q^5 + 3908 q^6 + 54474 q^7 + 93279 q^8 + \cdots, \\
    Y_{16, 2} &= q^4 + \frac{864}{25} q^5 + \frac{2736}{5} q^6 + \frac{188288}{35} q^7 + \frac{107998}{5} q^8 + \frac{1051008}{5} q^9 + \cdots.
\end{align*}

\bibliographystyle{plain} 
\bibliography{refs} 

\end{document}